\documentclass[10pt,twoside]{amsart}
\usepackage{amsmath,amsfonts,amssymb,amsxtra,bm,setspace,xspace,graphicx,lmodern,psfrag,epsfig,color,latexsym,stmaryrd,scalerel,stackengine,enumerate,subcaption,caption}
\usepackage[colorlinks=true]{hyperref}\hypersetup{urlcolor=blue, citecolor=red}

\numberwithin{equation}{section}
\newtheorem{thm}{Theorem}[section]

\newtheorem{lem}[thm]{Lemma}

\newtheorem{rem}[thm]{Remark}

\newtheorem{exm}[thm]{Example}

\stackMath
\newcommand\reallywidehat[1]{%
	\savestack{\tmpbox}{\stretchto{%
			\scaleto{%
				\scalerel*[\widthof{\ensuremath{#1}}]{\kern-.6pt\bigwedge\kern-.6pt}%
				{\rule[-\textheight/2]{1ex}{\textheight}}
			}{\textheight}%
		}{0.5ex}}%
	\stackon[1pt]{#1}{\tmpbox}%
}
\newcommand{\vertiii}[1]{{\left\vert\kern-0.25ex\left\vert\kern-0.25ex\left\vert #1 
		\right\vert\kern-0.25ex\right\vert\kern-0.25ex\right\vert}}

\newcommand{\R}{{\mathbb R}}

\begin{document}
	
	\title[DG method for Vlasov-Stokes']{Discontinuous Galerkin methods for\\ the Vlasov-Stokes' system }
	
	\bibliographystyle{alpha}
	
	\author[Harsha Hutridurga]{Harsha Hutridurga}
	\address{H.H.: Department of Mathematics, Indian Institute of Technology Bombay, Powai, Mumbai 400076 India.}
	\email{hutri@math.iitb.ac.in}
	
	\author[Krishan Kumar]{Krishan Kumar}
	\address{K.K.: Department of Mathematics, Indian Institute of Technology Bombay, Powai, Mumbai 400076 India.}
	\email{krishankumar@math.iitb.ac.in}
	
	\author[Amiya K. Pani]{Amiya K. Pani}
	\address{A.K.P.: Department of Mathematics, Birla Institute of Technology and Science, Pilani, KK Birla Goa Campus, NH 17 B, Zuarinagar, Goa 403726 India.}
	\email{amiyap@goa.bits-pilani.ac.in}
	
	\date{\today}
	
	\maketitle
	
	\thispagestyle{empty}

	\begin{abstract}
		This paper develops and analyses semi-discrete numerical method for two dimensional Vlasov-Stokes' system with periodic boundary condition. The method is  based on coupling of semi-discrete discontinuous Galerkin method for the Vlasov equation with discontinuous Galerkin scheme for the stationary incompressible Stokes' equation. The proposed method is both mass and momentum conservative. Since it is difficult to establish non-negativity of the discrete local density, the generalized discrete Stokes' operator become non-coercive and indefinite and under smallness condition on the discretization parameter,  optimal error estimates  are established with help of a modified the Stokes' projection to deal with Stokes' part and with the help of a special projection to tackle the Vlasov part. Finally, numerical experiments based on the dG method combined with a splitting algorithm are performed.
	\end{abstract}
	
	
	\section{Introduction}
	In this paper, we study a mathematical model describing sprays and aerosols, bubbly flows or suspension and sedimentation phenomena in two dimensions. It is a coupled system of $2$D incompressible steady generalized Stokes' equation and a Vlasov type equation:
	\begin{equation}\label{eq:continuous-model}
		\left\{
		\begin{aligned}
			\partial_t f + v\cdot \nabla_x f + \nabla_v \cdot \Big( \left( \bm{u} - v \right) f \Big) & = 0  \qquad \qquad \mbox{ in }(0,T)\times\Omega_x\times\R^2,
			\\
			f(0,x,v) & = f_{0}(x,v)\quad  \mbox{ in }\Omega_x\times\R^2.
		\end{aligned}
		\right.
	\end{equation}
	\begin{equation}\label{contstokes}
		\left\{
		\begin{aligned}
			- \Delta_x \bm{u} + \rho \bm{u} +\nabla_x p & = \rho V  \quad \mbox{ in }\Omega_x,
			\\
			\nabla_x \cdot \bm{u} & = 0  \qquad  \mbox{ in }\Omega_x,
		\end{aligned}
		\right.
	\end{equation}	
	Here, $\Omega_x$ denotes the two dimensional torus $\mathbb{T}^2$. The fluid velocity of the background medium is denoted by $\bm{u}(t,x)$, and the pressure is denoted by $p(t,x)$ which are functions of space and time variables. The distribution function $f(t,x,v)$ describing the particle distribution depends on the time variable $t > 0$, the spatial variable $x \in \Omega_x$ and the velocity variable $v \in \R^2$. The local density of the particles is given by $\rho := \int_{\R^2} \;f\,{\rm d}v$.	
The coupling between the two systems is due to a drag force which is proportional to the relative velocity $(\bm{u} - v)$. This system is also important in biological applications, such as transport in the respiratory tract \cite{gemci2002numerical,baranger2005modeling}.
	
	The Vlasov-Stokes' system has been rigorously derived from the dynamics of a system of particles in a fluid by Jabin and Perthame \cite{jabin2000notes}. The existence of solution with its large time behaviour of a similar system has been studied by Hamdache \cite{Hamdache_1998}. In this work, authors take unsteady incompressible Stokes' equation as the fluid equation for the background medium. The main idea of the existence proof in \cite{Hamdache_1998} is to first regularize the main problem followed by an application of the Schauder fixed point theorem to arrive at a solution of the regularized problem. Finally, the regularized solution in the limiting case yields the existence of a weak solution to the original problem. One of the crucial steps in the proof is the use of non-negativity of $\rho$, which makes the generalized Stokes' system into a damped Stokes' problem leading to an application of the inf-sup condition. Based on the Banach fixed point argument and using $\rho \geq 0$ in a crucial way, H\"ofer \cite{hofer2018} has proved the existence of a unique weak solution to the problem \eqref{eq:continuous-model}-\eqref{contstokes}. In \cite{our}, authors proved the existence of a unique strong solution in $3$D using the Banach fixed point argument. In \cite{our}, the authors consider the unsteady Stokes' equation for the background medium. For some regularity result, see \cite{gasser2000regularity}. 	 
	
In this paper, we develop discontinuous Galerkin methods for both the equations (dG-dG). Our work is inspired by the works \cite{ayuso2009discontinuous} and \cite{de2012discontinuous} which dealt with the Vlasov-Poisson system. Recently, in \cite{hutridurga2023discontinuous}, a discontinuous Galerkin method for the Vlasov-viscous Burgers' system is developed. In those three papers, the authors have analysed conservation and stability properties of the discrete system and have also derived optimal error estimates with the help of a special projection operator. The Vlasov equation is a form of transport problem which conserves total mass. Out of the several numerical methods, discontinuous Galerkin method has the property to preserve the total mass of the discrete system and hence, we propose dG for the Vlasov equation. One can also choose local discontinuous Galerkin or any conforming numerical method for the Stokes' system. 
One of the main difficulty in proving results like existence and uniqueness as well as optimal error estimates for the semi-discrete system
is the currently unavailable results demonstrating the non-negativity of the discrete local density $\rho_h$ and hence, the corresponding generalized discrete Stokes' problem becomes non-coercive and indefinite. This creates serious difficulties in the analysis and therefore, a more refined analysis is needed to derive results under smallness condition on the discretizing parameter $h$. In the literature, there have been some numerical works for similar type of models in \cite{goudon2013asymptotic,goudon2014asymptotic}.

Our major contributions in this paper are as follows:
	\begin{itemize}
		\item Conservation properties and stability of the discrete system are analysed.
		\item Since the generalized discrete Stokes' operator fails to satisfy the inf-sup condition, a more in depth analysis helps to achieve the result under a smallness assumption on the discretizing parameter $h$.  
		\item Based on a modified generalized Stokes' projection and a special projection, optimal error estimates are derived for the semi-discrete system.
		\item Existence of a unique pair of solutions for the discrete system is established.
		\item Finally, using a splitting algorithm, we perform some numerical experiments.
	\end{itemize}
	The paper is structured as follows. In section \ref{cts}, we describe some properties of the continuum model. In section \ref{dg}, we introduce a semi-discrete numerical method followed by proving certain properties of the discrete solution. The error analysis for the semi-discrete method is detailed in section \ref{err}. Finally, in section \ref{comp}, some computational results are given. 
	\section{Qualitative and Quantitative aspects of the model problem}\label{cts}
	
Through out this paper, we use standard notation for Sobolev spaces. Denote by $W^{m,p}$ the $L^p$-Sobolev space of order $m \geq 0$.
 Set $L^2_0(\Omega_x) = \left\{g \in L^2(\Omega_x): \int_{\Omega_x}g\,{\rm d}x = 0\right\},$\\$ \bm{W^{m,p}} = \left(W^{m,p}(\Omega_x)\right)^2, \|\bm{u}\|_{\bm{L^\infty}} = \|u\|_{[L^\infty(\Omega_x)]^2},\|\bm{u}\|_{\bm{W^{m,p}}} = \|u\|_{[W^{m,p}(\Omega_x)]^2}  \, \, \forall\,\, m \geq 0,\,\, 1 \leq p \leq \infty$. We, further, denote a special class of divergence-free (in the sense of distribution) vector fields by
\[
\bm{J_1} = \left\{\bm{z \in W^{1,2}}: \nabla_x\cdot \bm{z} = 0, \bm{z} \,\, \mbox{is periodic} \right\}.
\] 
%
Set the local macroscopic velocity $V$ as
\[
 V(t,x) = \frac{1}{\rho}\int_{\R^2} f(t,x,v)v\,{\rm d}v.
\]	
We define some physical quantities:  mass and momentum, respectively as
\[
\int_{\R^2}\int_{\Omega_x} f(t,x,v)\,{\rm d}x\,{\rm d}v \quad \mbox{and}\quad
\int_{\R^2}\int_{\Omega_x} vf(t,x,v)\,{\rm d}x\,{\rm d}v.
\]
The energy and the velocity moments $m_kf$, for $k \geq 0$, respectively, as
\[
\int_{\R^2}\int_{\Omega_x}  |v|^2\; f(t,x,v)\,{\rm d}x\,{\rm d}v \quad \mbox{and}\quad
m_kf(t,x) := \int_{\R^2}|v|^kf(t,x,v)\,{\rm d}v.
\]
Throughout this manuscript, any function defined on $\Omega_x$ is assumed to be periodic in the $x$-variable.

The following properties hold for the solution $(f,\bm{u},p)$ of the Vlasov-Stokes' equations \eqref{eq:continuous-model}-\eqref{contstokes}, whose proof can be found in \cite{Hamdache_1998}.
\begin{itemize}
	\item  \textbf{(Positivity preserving)} For any given non-negative initial datum $f_0$, the solution $f$ is also non-negative, i.e., if $f_0(x,v) \geq 0$ then $f(t,x,v) \geq 0$. Further, the local density $\rho(t,x) \geq 0$.
	\item \textbf{(Mass conservation)} The distribution function $f$ conserves mass in the sense that
	\begin{equation}\label{contmass}
		\int_{\Omega_x}\int_{\R^2} f(t,x,v)\,{\rm d}v\,{\rm d}x = \int_{\Omega_x}\int_{\R^2} f_0(x,v)\,{\rm d}v\,{\rm d}x \quad \forall \,\, t \in [0,T].
	\end{equation}
	\item \textbf{(Total momentum conservation)} The distribution function $f$ conserves momentum in the sense that
	\begin{equation}\label{contmomentum}
		\int_{\Omega_x}\int_{\R^2} v\,f(t,x,v)\,{\rm d}v\,{\rm d}x = \int_{\Omega_x}\int_{\R^2} v\,f_0(x,v)\,{\rm d}v\,{\rm d}x \quad \forall \,\, t \in [0,T].
	\end{equation}
	\item \textbf{(Energy dissipation)} The energy of the system dissipates, i.e.,
	\begin{equation}\label{contenergy}
		\frac{{\rm d}}{{\rm d}t}\int_{\R^2}\int_{\Omega_x}|v|^2f(t,x,v)\,{\rm d}x\,{\rm d}v \leq 0, \quad \mbox{for any }\quad f_0(x,v) \geq 0.
	\end{equation}
\end{itemize} 

Multiplying equation \eqref{eq:continuous-model} by $|v|^2$, and equation \eqref{contstokes} by $2\bm{u}$ followed by an integration by parts shows
\begin{equation*}\label{ene}
	\begin{aligned}
		\int_{\R^2}\int_{\Omega_x} |v|^2f(t,x,v)\,{\rm d}x\,{\rm d}v &+ 2\int_0^t\int_{\Omega_x}|\nabla_{x}\bm{u}|^2\,{\rm d}x\,{\rm d}t 
		\\
		&+ 2\int_0^t\int_{\R^2}\int_{\Omega_x}|\bm{u} - v|^2f\,{\rm d}x\,{\rm d}v\,{\rm d}t = \int_{\R^2}\int_{\Omega_x} |v|^2\,f_0\,{\rm d}x\,{\rm d}v.
	\end{aligned}  
\end{equation*}
This is referred to as the energy identity.
As $f \geq 0$, the third term on the left hand side becomes non-negative and hence
\begin{equation*}
	\begin{aligned}
		\int_{\R^2}\int_{\Omega_x} |v|^2f(t,x,v)\,{\rm d}x\,{\rm d}v + 2\int_0^t\int_{\Omega_x}|\nabla_{x}\bm{u}|^2\,{\rm d}x\,{\rm d}t \leq \int_{\R^2}\int_{\Omega_x} |v|^2\,f_0\,{\rm d}x\,{\rm d}v.
	\end{aligned}  
\end{equation*}
Moreover, if $\int_{\Omega_x}\int_{\R^2}\,|v|^2f_0(x,v)\,{\rm d}v\,{\rm d}x < \infty$, then the above inequality yields
\begin{equation}
	\bm{u} \in L^2(0,T;\bm{J_1}).
\end{equation}
A use of the Sobolev inequality yields
\begin{equation}
	\bm{u} \in L^2(0,T;\bm{L^p}), \quad \forall \,\, 2 \leq p < \infty. 
\end{equation}


The following result shows integrability estimates on the local density and on the momentum. As these appear as source terms in the Stokes' equation, these estimates are crucial in deducing the regularity of solution to the Stokes' problem. The proof of the following result can be found in \cite[Lemma 2.2, p.56]{Hamdache_1998}.
\begin{lem}\label{density}
	Let $p \geq 1$. Let $\bm{u} \in L^2(0,T;\bm{L^{p+2}}), f_0 \in L^\infty(\Omega_x \times \R^2)\cap L^1(\Omega_x \times \R^2)$ and let
	\[
	\int_{\R^2}\int_{\Omega_x}\,|v|^pf_0\,{\rm d}x{\rm d}v <\infty.
	\] 
	Then, the local density $\rho$ and the momentum $\rho V$ satisfy the following:
	\begin{equation}\label{rhos}
		\rho \in L^\infty\left(0,T;L^\frac{p+2}{2}(\Omega_x)\right) \quad \mbox{and} \quad \rho V \in L^\infty\left(0,T; L^\frac{p+2}{3}(\Omega_x)\right).
	\end{equation}
\end{lem}

\begin{rem}
	Taking $p = 4$ in Lemma \ref{density} yields 
	\begin{equation}\label{rhos1}
		\rho \in L^\infty\left(0,T;L^3(\Omega_x)\right) \quad \mbox{and} \quad \rho V \in L^\infty\left(0,T; L^2(\Omega_x)\right).
	\end{equation}
	These integrability properties of $\rho$ and $\rho V$ in the Stokes' equation \eqref{contstokes} guarantee that $\bm{u} \in L^1(0,T;\bm{L^\infty})$ which plays a crucial role in arriving at an $L^\infty$ estimate on $\rho$.
\end{rem}
The following Lemma gives the $L^\infty$ estimate in time and space variable for the local density. The proof of this can be found in \cite[Proposition 4.6, p.44]{han2019uniqueness}. 
\begin{lem}\label{lem:density}
	Let $\bm{u} \in L^1(0,T;\bm{L^\infty})$ and let $\sup_{C^r_{t,v}}f_0 \in L^\infty_{loc}\left(\R_+;L^1(\R^2)\right)$, where $C^r_{t,v} := \Omega_x \times B(e^tv,r), \, \forall\, r > 0$. $B(e^tv,r)$ denote a ball of radius $r$ with center at $e^{t}v$. Then, the following estimate holds:
	\begin{equation*}
		\|\rho(t,x)\|_{L^{\infty}((0,T)\times\Omega_x)} \leq e^{2T}\sup_{t \in [0,T]}\|\sup_{C^r_{t,v}}f_0\|_{L^1(\R^2)}.
	\end{equation*}     	
\end{lem}

	\section{Discontinuous Galerkin approximations}\label{dg}
	In this section, we design and analyze a discontinuous Galerkin scheme to approximate solutions to the continuum model \eqref{eq:continuous-model}-\eqref{contstokes}. Note that a compactly supported (in the $v$ variable) initial datum results in a compactly supported (in the $v$ variable) solution $f(t,x,v)$ to the Vlasov equation \eqref{eq:continuous-model} provided $t \in [0,T]$ where $T$ depends on the size of $f_0$ and the magnitude of the background fluid velocity $\bm{u}$. So, if we restrict ourselves to compactly supported (in the $v$ variable) data, we can assume without loss of generality that there exist constants $L > 0$ and $T > 0$ such that $\mbox{supp}f(t,x,\cdot) \subset \Omega_v =: [-L,L]^2$ for all $t \in [0,T]$ and $x \in \Omega_x$.     
	
	\noindent
	Let $\mathcal{T}^x_{h}$ and $\mathcal{T}^v_{h}$ be two shape regular and quasi-uniform families of Cartesian partitions of $\Omega_x$ and $\Omega_v$, respectively. Let $\{\mathcal{T}_{h}\}$ be defined as the Cartesian product of these two partitions, i.e.,
	\begin{align*}
		\mathcal{T}_{h} = \left\{ R = T^x \times T^v : T^x \in \mathcal{T}^x_{h}, \,  T^v \in \mathcal{T}^v_{h} \right\}.
	\end{align*}
	The mesh sizes $h$, $h_x$ and $h_v$ relative to these partitions are defined as follows:
	\[
	h_x := \max_{T^x \in \mathcal{T}^x_{h}} {\rm diam} (T^x);
	\quad
	h_v := \max_{T^v \in \mathcal{T}^v_{h}} {\rm diam} (T^v);
	\quad
	h := \max \left( h_x, h_v \right).
	\]
	We denote by $\Gamma_x$ and $\Gamma_v$ the set of all edges of the partitions $\mathcal{T}^x_{h}$ and $\mathcal{T}^v_{h}$, respectively and we set $\Gamma = \Gamma_x \times \Gamma_v$. The sets of interior and boundary edges of the partition $\mathcal{T}^x_{h}$ (respectively, $\mathcal{T}^v_{h}$) are denoted by $\Gamma^0_{x}$ (respectively, $\Gamma^0_{v}$) and $\Gamma^\partial_{x}$ (respectively, $\Gamma^\partial _{v}$), so that $\Gamma_x = \Gamma^0_{x} \cup \Gamma^\partial _{x}$ (respectively, $\Gamma_v = \Gamma^0_{v} \cup \Gamma^\partial _{v}$).
	
Our objective is to develop a discontinuous Galerkin scheme to approximate solutions to the continuum model. We will consider the following broken polynomial spaces:
	\begin{equation*}
		\begin{aligned}
			X_h &:= \left\{ \psi \in L^2(\Omega_x) : \psi\vert_{T^x} \in \mathbb{P}^k(T^x),\, \,\forall\, T^x \in \mathcal{T}_h^x \right\},
			\\
			V_h &:= \{\psi \in L^2(\Omega_v) : \psi\vert_{T^v} \in \mathbb{P}^k(T^v),\, \,\forall\, T^v \in \mathcal{T}_h^v\},
			\\
			\mathcal{Z}_h &:= \left\{ \phi \in L^2(\Omega_x\times\Omega_v) : \phi\vert_{R} \in \mathbb{P}^k(T^x) \times \mathbb{P}^k(T^{v}),\,  \forall R = T^x \times T^v \in \mathcal{T}_{h} \right\},
			\\
			U_h &:= \left\{ \bm{\psi} \in \bm{L^2} : \bm{\psi}\vert_{T^x} \in \left(\mathbb{P}^k(T^x)\right)^2,\, \,\forall\, T^x \in \mathcal{T}_h^x \right\},
			\\
			P_h &:= \left\{ \psi \in L^2_0(\Omega_x) : \psi\vert_{T^x} \in \mathbb{P}^k(T^x),\, \,\forall\, T^x \in \mathcal{T}_h^x \right\},
		\end{aligned}
	\end{equation*}	
	where $\mathbb{P}^k(T)$ is the space of scalar polynomials of degree at most $k$ in each variable. More precisely, we will approximate the background velocity $\bm{u}$ by $\bm{u_h} \in U_h$, the fluid pressure $p$ by $p_h \in P_h$, the distribution function $f$ by $f_h \in \mathcal{Z}_h$. 
	
 Let $\bm{n}^-$ and $\bm{n}^+$ be the outward and inward unit normal vectors on the element $T$. The interior trace $\psi^-$ of $\psi$ and the outer trace $\psi^+$ of $\psi$ on $\Gamma_x$ is defined by 
 \[
 \psi^{\pm}(x,\cdot) = \lim_{\epsilon \to 0}\psi(x\pm \epsilon \bm{n}^-,\cdot) \hspace{5mm}\forall \hspace{1mm} x \in \Gamma_x.
 \] 
 Following \cite{arnold2002unified,vemaganti2007discontinuous}, we set the average and jump of a scalar function $\psi$ and a vector-valued function $\bm{\Psi}$ at the edges as follows:
	\[
	\{\psi \} = \frac{1}{2}(\psi^- + \psi^+), \hspace{5mm} [\![\psi]\!] = \psi^- \bm{n}^- + \psi^+ \bm{n}^+ \hspace{5mm} \text{on} \hspace{1mm}  \Gamma^0_r, \hspace{3mm} r = x \hspace{1mm}\text{or}\hspace{1mm} v\, ,
	\]
	\[
	\{\bm{\Psi} \} = \frac{1}{2}(\bm{\Psi}^- + \bm{\Psi}^+), \hspace{5mm} [\![\bm{\Psi}]\!] = \bm{\Psi}^-\cdot \bm{n}^- + \bm{\Psi}^+\cdot \bm{n}^+ \hspace{5mm} \text{on} \hspace{1mm}  \Gamma^0_r, \hspace{3mm} r = x \hspace{1mm}\text{or}\hspace{1mm} v.
	\]
	For $0 \leq \delta \leq 1$ the weighted average of a vector valued function $\bm{\Psi}$ is defined by
	\[
	\{\bm{\Psi}\}_\delta := \delta \bm{\Psi}^+ + (1 - \delta)\bm{\Psi}^-.
	\]
	Note that for a fixed edge $\bm{n}^- = - \bm{n}^+$. For boundary edges, the jump and average are taken to be $[\![\psi]\!] = \psi\bm{n}$ and $\{\psi\} = \psi$.

	\noindent
	We will be considering the following discrete spaces:
	\[
	W^{m,p}(\mathcal{T}_h) = \left\{ \psi \in L^{p}(\Omega) : \psi_{\mid_R} \in W^{m,p}(R) \quad \forall R \in \mathcal{T}_h\right\} \qquad m\geq 0,\quad 1\leq p \leq \infty.
	\]
	We further denote $W^{m,2}(\mathcal{T}_h)$ by $H^m(\mathcal{T}_h)$ for $m \geq 1$.
	We define the following discrete norms: 
	\[
	 \|w\|_{m,\mathcal{T}_h} = \left(\sum_{R \in \mathcal{T}_h}\|w\|^2_{m,R}\right)^\frac{1}{2} \quad \forall\,w \in H^m(\mathcal{T}_h),\,\, m \geq 0
	\]
	\[
\|w\|_{L^p(\mathcal{T}_h)} = \left(\sum_{R \in \mathcal{T}_h}\|w\|^p_{L^p(R)}\right)^\frac{1}{p},\,\,\forall\,w \in L^p(\mathcal{T}_h), 
	\]
	for all $1 \leq p < \infty$ and $\|w\|_{L^\infty(\mathcal{T}_h)} = \mbox{esssup}_{w \in \mathcal{T}_h}\vert w \vert$. For $F \in \Gamma_x, w_h \in X_h$
	\[
	\|w_h\|^2_{0,F} = \int_{F}[\![w_h]\!]_F\cdot[\![w_h]\!]_F\,{\rm d}s(x).
	\]
	 For all $\left(\bm{w_h},q_h\right) \in U_h \times P_h$,
	\[
	\vertiii{\bm{w_h}}^2 = \|\nabla \bm{w_h}\|^2_{0,\mathcal{T}_h^x} + \sum_{F \in \Gamma_x}h^{-1}_x\|[\![\bm{w_h}]\!]\|^2_{L^2(F)},
	\]
	\[
	\vertiii{\left(\bm{w_h},q_h\right)}^2 = \vertiii{\bm{w_h}}^2 + \|q_h\|_{0,\mathcal{T}_h^x}^2 + \sum_{F \in \Gamma_x}h_x\|[\![q_h]\!]\|^2_{\bm{L^2}}.
	\]
	
	 Next, we recall certain function inequalities in the aforementioned discrete function spaces. These inequalities will be used in our subsequent analysis. 
		\begin{itemize}
			\item \textbf{Inverse inequality:} (see \cite[Lemma 1.44, p. 26]{2}) For any $w_h  \in X_h$, there exists a constant $C > 0$ such that 
			\begin{equation}\label{eq:inverse}
				\|\nabla_xw_h\|_{0,T^x} \leq Ch_x^{-1}\|w_h\|_{0,T^x} \quad \forall \,\, T^x \in \mathcal{T}_h^x.
			\end{equation}
			\item \textbf{Trace inequality:} (see \cite[Lemma 1.46, p. 27]{2}) If $w_h \in X_h$, then for all $F \in \Gamma_x$ and $T^x \in \mathcal{T}_h^x$ we have
			\begin{equation}\label{traceeqn}
				\|w_h\|_{0,F } \leq Ch_x^{-\frac{1}{2}}\|w_h\|_{0,T^x}.
			\end{equation}
			\item \textbf{Norm comparison:} (see \cite[Lemma 1.50, p. 29]{2}) Let $1\leq p,q \leq \infty$ and $w_h \in X_h$. Then, for all $T^x \in \mathcal{T}_h^x$
			\begin{equation}\label{normcompeqn}
				\|w_h\|_{L^p(T^x)} \leq Ch_x^{\frac{2}{p}-\frac{2}{q}}\|w_h\|_{L^q(T^x)}.
			\end{equation}
			\item A \textbf{Poincar\'e-Friedrichs} type inequality is proved in \cite[Lemma 2.1, p. 744]{arnold1982interior} which says that
			\[
			\|v_h\|_{0,\Omega_x} \leq C\vertiii{v_h}, \quad \forall \quad v_h \in H^1(\Omega_x).
			\]
			\item \textbf{Projection operators :} Let $k \geq 0$. Let $\mathcal{P}_x : L^2(\Omega_x) \rightarrow X_h$, $\bm{\mathcal{P}_x} : \bm{L^2} \rightarrow U_h, \mathcal{P}_v : L^2(\Omega_v) \rightarrow V_h$ and $\mathcal{P}_h : L^2(\Omega) \rightarrow \mathcal{Z}_h$ be the standard $L^2$-projections onto the spaces $X_h, U_h, V_h$ and $\mathcal{Z}_h$, respectively. Note that $\mathcal{P}_h = \mathcal{P}_x \otimes \mathcal{P}_v,$ (see \cite{ciarlet2002finite,agmon2010lectures}).\\
			By definition, $\mathcal{P}_h$ is stable in $L^2$ and it can be further shown to be stable in all $L^p$- norms (see \cite{crouzeix1987stability} for details). Let $1 \leq p \leq \infty$. Then for any $w \in L^p(\Omega)$
			\begin{equation}\label{L2stability}
				\|\mathcal{P}_h(w)\|_{L^p(\mathcal{T}_h)} \leq C\|w\|_{L^p(\Omega)}.
			\end{equation}
		\end{itemize}
	
Next, we give a result which relates $L^\infty$ bounds with $L^2$ bounds while approximating functions in aforementioned broken polynomial spaces. 
	\begin{lem}\label{L:uinf}
	Let $g \in W^{1,\infty}(\Omega_x) \cap H^{k+1}(\Omega_x)$. Then,
	\begin{equation*}\label{inftybound}
		\|g - g_h\|_{L^\infty(\mathcal{T}_h^x)} \lesssim h_x\|g\|_{W^{1,\infty}(\Omega_x)} + h^k_x\|g\|_{H^{k+1}(\Omega_x)} + h_x^{-1}\|g - g_h\|_{0,\mathcal{T}_h^x} \quad \forall \quad g_h \in X_h.
	\end{equation*}
\end{lem}

\begin{proof}
	Observe that
	\begin{align*}
		\|g - g_h\|_{L^\infty(\mathcal{T}_h^x)} &\leq \|g - \mathcal{P}_xg\|_{L^\infty(\mathcal{T}_h^x)} + \|\mathcal{P}_xg - g_h\|_{L^\infty(\mathcal{T}_h^x)}
		\\
		&\lesssim h_x\|g\|_{W^{1,\infty}(\Omega_x)} + h^{-1}_x\|\mathcal{P}_xg - g_h\|_{0,\mathcal{T}_h^x}
		\\
		&\lesssim h_x\|g\|_{W^{1,\infty}(\Omega_x)} + h^{-1}_x\|g - \mathcal{P}_xg\|_{0,\mathcal{T}_h^x} + h^{-1}_x\|g - g_h\|_{0,\mathcal{T}_h^x}
		\\
		&\lesssim h_x\|g\|_{W^{1,\infty}(\Omega_x)} + h^k_x\|g\|_{H^{k+1}(\Omega_x)} + h^{-1}_x\|g - g_h\|_{0,\mathcal{T}_h^x}.
	\end{align*}
Here, in the second step, we have employed the projection estimate \cite{wahlbin2006superconvergence} for the first term and the norm comparison estimate \eqref{normcompeqn} is employed for the second term. In the last step, the projection estimate is used for the second term.
\end{proof}


\begin{rem}
	A similar result to Lemma \ref{L:uinf} also hold for functions on the phase space $\Omega$. More precisely, let $\tilde{g} \in W^{1,\infty}(\Omega)\cap H^{k+1}(\Omega)$. Then, 
	\begin{equation*}
		\|\tilde{g} - \tilde{g}_h\|_{L^\infty(\mathcal{T}_h)} \lesssim h\|\tilde{g}\|_{W^{1,\infty}(\Omega)} + h^k\|\tilde{g}\|_{H^{k+1}(\Omega)} + h^{-1}\|\tilde{g} - \tilde{g}_h\|_{0,\mathcal{T}_h} \quad \forall \quad \tilde{g}_h \in \mathcal{Z}_h.
	\end{equation*}
\end{rem}

	\subsection{Semi-discrete dG formulation}
	
Our approximations scheme is to find $(f_h,\bm{u_h},p_h)(t) \in \mathcal{Z}_h\times U_h \times P_h$, for $t \in [0,T]$ such that\\
	\begin{equation}{\label{bh}}
		\left(\partial_t f_h,\psi_h\right) + \sum_{R \in \mathcal{T}_h}\mathcal{B}_{h,R}(\bm{u_h} ; f_{h},\psi_{h}) = 0  \hspace{2mm} \forall \,\,\psi_{h} \in \mathcal{Z}_h,
	\end{equation}
\begin{equation}\label{ch1}
	a_h(\bm{u_h},\bm{\phi_h}) + b_h\left(\bm{\phi_h},p_h\right) + \left(\rho_h\bm{u_h},\bm{\phi_h}\right) = \left(\rho_hV_h, \bm{\phi_h}\right) \quad \forall\,\, \bm{\phi_h} \in U_h,
\end{equation}
\begin{equation}\label{ch2}
	-b_h(\bm{u_h},q_h) + s_h(p_h,q_h) = 0 \quad \forall\,\,q_h \in P_h,
\end{equation}
	with the initial datum $f_h(0) = \mathcal{P}_h(f_0) \in \mathcal{Z}_h$. Here, similar to the continuum setting, we have set the discrete local density $\rho_h$ and the discrete local macroscopic velocity $V_h$ associated with $f_h$ as follows: 
	\begin{equation}\label{rh}
		\rho_h = \sum_{T^v \in \mathcal{T}^v_{h}}\int_{T^v}f_h\,{\rm d}v, \quad \mbox{and}\quad	V_h = \frac{1}{\rho_h}\sum_{T^v \in \mathcal{T}^v_{h}}\int_{T^v}vf_h\,{\rm d}v.
	\end{equation}
In the above formulation, \eqref{bh} corresponds to the Vlasov equation and \eqref{ch1}-\eqref{ch2} correspond to the incompressible Stokes' equations.\\
In \eqref{bh}, we have used the following notation:
	\begin{equation}\label{bhdef}
		\begin{aligned}
			\mathcal{B}_{h,R}(\bm{u_h};f_h,\psi_h) & :=  - \int_{R}f_hv\cdot\nabla_x\psi_h \, {\rm d}v\, {\rm d}x + \int_{T^v}\int_{\partial T^x}\reallywidehat{v\cdot\bm{n}f_h}\psi_h\, {\rm d}s(x)\, {\rm d}v
			\\
			& \quad - \int_{R}f_h(\bm{u_h} - v)\cdot\nabla_v\psi_h \, {\rm d}v\, {\rm d}x +\int_{T^x}\int_{\partial T^v} \reallywidehat{(\bm{u_h} - v)\cdot\bm{n}f_h}\psi_h \, {\rm d}s(v)\, {\rm d}x
		\end{aligned}
	\end{equation}
wherein the numerical fluxes are taken to be
\begin{equation}\label{flux}
		\left\{
		\begin{aligned}
			\widehat{v\cdot \bm{n}f_h} &= \{vf_h\}_\alpha\cdot\bm{n} := \left(\{vf_h\} + \frac{|v\cdot \bm{n}|}{2}[\![f_h]\!]\right)\cdot\bm{n}  \hspace{5mm} \text{on}\hspace{1mm} \Gamma^0_{x} \times T^v,
			\\
			\reallywidehat{(\bm{u_h} - v)\cdot \bm{n}f_h} &= \{(\bm{u_h} - v)f_h\}_\beta\cdot\bm{n} :=  \left(\{(\bm{u_h} - v)f_h\} + \frac{|(\bm{u_h} - v)\cdot \bm{n}|}{2}[\![f_h]\!]\right)\cdot \bm{n}  \hspace{3mm} \text{on}\hspace{1mm} T^x \times \Gamma^0_{v},
		\end{aligned}
		\right.
	\end{equation}
with $\bm{n} = \bm{n}^-, \alpha = \frac{1}{2}(1\pm\mbox{sign}(v\cdot\bm{n}^\pm))$ and $\beta = \frac{1}{2}(1\pm\mbox{sign}((\bm{u_h}-v)\cdot\bm{n}^\pm))$. For the more details about weighted average refer \cite{brezzi2004discontinuous}. On the boundary edges $e \in \Gamma_r^\partial, r = x,v$, we impose periodicity for $\reallywidehat{v\cdot\bm{n}f_h}$ and the compactness of support for $\reallywidehat{\left(\bm{u_h} - v\right)\cdot\bm{n}f_h}$. \\
In \eqref{ch1}, the bilinear form $a_h(\bm{u_h},\bm{\phi_h})$ stands for 
	\begin{equation*}{\label{aih}}
		\begin{aligned}
			a_h(\bm{u_h},\bm{\phi_h})  &= \sum_{i=1}^2a_{h,i}(u_{h,i},\phi_{h,i})
		\end{aligned}
\end{equation*}
where, $u_{h,1}, u_{h,2}$ and $\phi_{h,1}, \phi_{h,2}$ denote the Cartesian components of $\bm{u_h}$ and $\bm{\phi_h}$, respectively with
\begin{equation*}
	\begin{aligned}
		a_{h,i}(u_{h,i},\phi_{h,i}) &= \sum_{T^x\in \mathcal{T}^x_{h}}\int_{T^x}\nabla{u_{h,i}}\cdot\nabla\phi_{h,i}\,{\rm d}x + \sum_{F\in \Gamma_x}\int_F\frac{\vartheta}{h_x}[\![u_{h,i}]\!]\cdot[\![\phi_{h,i}]\!]\,{\rm d}s(x)
			\\
			& -\sum_{F\in \Gamma_x}\int_F\left(\{\nabla u_{h,i}\}\cdot [\![\phi_{h,i}]\!] + [\![u_{h,i}]\!]\cdot\{\nabla\phi_{h,i}\}\right)\,{\rm d}s(x).
		\end{aligned}
	\end{equation*} 	
Here, $\vartheta > 0$ is a penalty parameter. In \eqref{ch1}-\eqref{ch2}, $b_h(\cdot,\cdot)$ stands for
	\begin{equation}{\label{bhh}}
		\begin{aligned}
			b_h(\bm{u_h},q_h) = -\sum_{T^x\in \mathcal{T}^x_{h}}\int_{T^x} q_h\nabla\cdot  \bm{u_h}\,{\rm d}x + \sum_{F\in \Gamma_x}\int_F[\![\bm{u_h}]\!]\{q_h\}\,{\rm d}s(x)
		\end{aligned}
	\end{equation}	
and	the stabilization term in \eqref{ch2} is given by
	\begin{equation}\label{sh}
		\begin{aligned}
			s_h(p_h,q_h) = \sum_{F\in\ \Gamma_x}h_x\int_F[\![p_h]\!]\cdot[\![q_h]\!]\,{\rm d}s(x).
		\end{aligned}
	\end{equation}	
	
Note that \eqref{ch1}-\eqref{ch2} is equivalent to 
\begin{equation}\label{ch}	
	\mathcal{A}_{h}((\bm{u_h},p_h),(\bm{\phi_h},q_h))\, + \,\left(\rho_h\bm{u_h}, \bm{\phi_h}\right) = \left(\rho_h V_h, \bm{\phi_h}\right) \hspace{5mm} \forall\hspace{1mm}(\bm{\phi_h},q_h) \in U_h\times P_h,
\end{equation}
where
\begin{equation}\label{A_h}
	\mathcal{A}_{h}((\bm{u_h},p_h),(\bm{\phi_h},q_h)) = a_h(\bm{u_h,\phi_h}) + b_h(\bm{\phi_h},p_h) - b_h(\bm{u_h},q_h) + s_h(p_h,q_h).
\end{equation}	
We call \eqref{ch} as the discrete generalized Stokes' system due to the presence of the term involving $\rho_h$.

	It turns out that the bilinear form $a_h(\bm{u_h},\bm{\phi_h})$ is coercive in $\vertiii{\cdot}$-norm (for proof, see \cite[Lemma 4.12, p. 129]{2}), i.e. there exists a constant $\beta > 0$, independent of $h$, such that
	\begin{equation}\label{acoercive}
		\beta\vertiii{\bm{u_h}}^2 \leq a_h(\bm{u_h},\bm{u_h}).
	\end{equation}	
	Furthermore, the bilinear form $\mathcal{A}_h((\bm{u_h},p_h),(\bm{\phi_h},q_h))$ satisfies the discrete inf-sup stability condition and is bounded in the $\vertiii{(\cdot,\cdot)}$-norm, i.e., there exist constants $\alpha > 0$ and $C > 0$, independent of $h$, such that
	\begin{equation}\label{dic1}
		\alpha\vertiii{(\bm{u_h},p_h)} \leq \sup_{(\bm{\phi_h},q_h) \in U_h\times P_h\setminus\{(0,0)\}}\frac{\mathcal{A}_h((\bm{u_h},p_h),(\bm{\phi_h},q_h))}{\vertiii{(\bm{\phi_h},q_h)}}
	\end{equation}
	(see \cite[Lemma 6.13, p. 253]{2}) and 
	\begin{equation}\label{ahbd}
		|\mathcal{A}_h((\bm{u_h},p_h)(\bm{\phi_h},q_h))| \leq C\vertiii{(\bm{u_h},p_h)}\vertiii{(\bm{\phi_h},q_h)}.
	\end{equation}

\begin{rem}
	At present we are unable to show that the discrete local density is non-negative, i.e. $\rho_h \geq 0$. Hence we are unable to demonstrate that the bilinear form on the left hand side of the discrete generalized Stokes' system \eqref{ch}, that is,
	$$\mathcal{A}_{h}((\bm{u_h},p_h),(\bm{\phi_h},q_h))\, + \,\left(\rho_h\bm{u_h}, \bm{\phi_h}\right)$$
	satisfies the inf-sup condition. This makes the existence proof more involved. Hence, in sub section \ref{sec:exst}, we revisit the existence and  uniqueness result for the discrete system, where we show that \eqref{bh} and \eqref{ch} has a unique solution provided $h$ is small enough.
\end{rem}

	\subsection{Qualitative properties of the discrete solution}\label{dgp}
	
\begin{lem}\label{lem:distcon}
	Let  $(f_h,\bm{u_h},p_h) \in C^1([0,T];\mathcal{Z}_h\times U_h\times P_h)$ be the dG-dG approximation obtained by solving the discrete system \eqref{bh} and \eqref{ch} with initial datum $f_h(0) = \mathcal{P}_h(f_0)$. Then, for all $t  \geq 0$,
	\begin{equation}\label{distmass}
		 \int_\Omega f_h(t,x,v)\,{\rm d}x\,{\rm d}v = \int_\Omega f_0(x,v)\,{\rm d}x\,{\rm d}v \quad  \mbox{when} \quad k \geq 0
	\end{equation}
\begin{equation}\label{distmomentum}
	\int_\Omega vf_h(t,x,v)\,{\rm d}x\,{\rm d}v = \int_\Omega vf_0(x,v)\,{\rm d}x\,{\rm d}v\quad  \mbox{when} \quad k \geq 1
\end{equation}
\begin{equation}\label{distenergy}
		\begin{aligned}
			\frac{1}{2}\int_\Omega\vert v\vert^2 f_h(t,x,v)\,{\rm d}x\,{\rm d}v + \int_0^t a_h(\bm{u_h}(s),\bm{u_h}(s))\,{\rm d}s + \int_0^ts_h(p_h(s),p_h(s))\,{\rm d}s 
			\\+ \int_0^t\int_\Omega\vert\bm{u_h}(s) - v\vert^2f_h(s,x,v)\,{\rm d}x\,{\rm d}v\,{\rm d}s 
			\\= \frac{1}{2}\int_\Omega\vert v\vert^2f_0(x,v)\,{\rm d}x\,{\rm d}v \quad \mbox{when} \quad k \geq 2.
		\end{aligned}
	\end{equation} 
\end{lem}	

The identities \eqref{distmass} and \eqref{distmomentum} are the discrete version of the mass and momentum conservation properties \eqref{contmass} and \eqref{contmomentum}, respectively, that hold true in the continuum setting . The identity \eqref{distenergy} is the discrete version of the energy identity \eqref{contenergy} from the continuum setting. However, unlike the continuum setting, we are unable to prove that the discrete system \eqref{bh} and \eqref{ch} has the positive preserving property. Hence, we do not, at present, have a discrete version of energy dissipation \eqref{contenergy}. 

The proof of the identities in the Lemma \ref{lem:distcon} are quite straightforward. A similar set of identities were proved for a dG scheme approximating solutions to the Vlasov-viscous Burgers' system. Please consult \cite[Lemmas $3.3 - 3.5$]{hutridurga2023discontinuous} for proofs.   

	Next, we prove a stability estimate for the dG scheme \eqref{bh} approximating solutions to the Vlasov equation.
	
	\begin{lem}\label{fhbd}
		Let $k \geq 0$ and let $f_h$ be the dG approximation to $f,$ satisfying $\eqref{bh}.$ Then, 
		\begin{equation*}
			\max_{t \in [0,T]}\|f_h\|_{0,\mathcal{T}_h} \leq e^{T}\|f_0\|_{0,\mathcal{T}_h} \quad \mbox{for any} \quad T > 0. 
		\end{equation*}
	\end{lem}
	
	\begin{proof}
		Choosing $\psi_h = f_h$ in \eqref{bh} yields
		\begin{equation*}
			\begin{aligned}
				&\frac{1}{2}\frac{{\rm d}}{{\rm d}t}\|f_h\|^2_{0,\mathcal{T}_h} - \sum_{R \in \mathcal{T}_h}\left( \frac{1}{2}\int_{R}\left(v\cdot \nabla_xf_h^2 + (\bm{u_h} - v)\cdot\nabla_vf_h^2\right)\,{\rm d}x\,{\rm d}v\right) 
				\\
				&-\sum_{T^v \in \mathcal{T}_h^v}\int_{T^v}\int_{\Gamma_x}\{vf_h\}_\alpha\cdot[\![f_h]\!]\,{\rm d}s(x)\,{\rm d}v - \sum_{T^x \in \mathcal{T}^x_{h}}\int_{T^x}\int_{\Gamma_v}\{(\bm{u_h} - v)f_h\}_\beta\cdot[\![f_h]\!]\,{\rm d}s(v)\,{\rm d}x = 0.
			\end{aligned}
		\end{equation*}
		After applying integration by parts in the second and in the third terms and after using the identity $[\![f_h^2]\!] = 2\{f_h\}[\![f_h]\!]$ with periodic boundary condition in $x$ and compact support in $v$, it follows that
		\begin{equation}\label{aa}
			\begin{aligned}
				&\frac{1}{2}\frac{{\rm d}}{{\rm d}t}\|f_h\|^2_{0,\mathcal{T}_h} - \|f_h\|^2_{0,\mathcal{T}_h} + \sum_{T^v \in \mathcal{T}^v_{h}}\int_{T^v}\int_{\Gamma_x^0}\frac{\mid v\cdot \bm{n}\mid}{2}[\![f_h]\!]\cdot[\![f_h]\!]\,{\rm d}s(x)\,{\rm d}v 
				\\
				&+ \sum_{T^x \in \mathcal{T}^x_{h}}\int_{T^x}\int_{\Gamma_v^0}\frac{\mid (\bm{u_h} - v)\cdot \bm{n}\mid}{2}[\![f_h]\!]\cdot[\![f_h]\!]\,{\rm d}s(v)\,{\rm d}x = 0.
			\end{aligned}
		\end{equation}
	Observe that the last two terms on the left hand side of the above equality are non-negative and hence can be dropped. An integration in time yields the desired result.
	\end{proof}
	\noindent
	As a consequence of Lemma \ref{fhbd}, it follows 
	\begin{equation}\label{rho1}
		\|\rho_h\|_{0,\mathcal{T}^x_{h} }\leq C(T)L\|f_0\|_{0,\mathcal{T}_h}, \quad \mbox{and} \quad \|\rho_hV_h\|_{0,\mathcal{T}^x_{h} }\leq C(T)L^2\|f_0\|_{0,\mathcal{T}_h}.
	\end{equation}
	
	
	\noindent
	For our subsequent use, we need the following lemma.
	\begin{lem}\label{rhoh2}
		Let $f_h$ be the dG approximation to the continuum solution $f$, obtained by solving \eqref{bh}. Let $\rho_h, V_h$ be the discrete local density, the discrete local macroscopic velocity associated with $f_h$ defined as in \eqref{rh}. Let $\rho, V$ be the local density, the local macroscopic velocity associated with $f$. Then   
		\begin{equation*}
			\begin{aligned}
				\|\rho - \rho_h\|_{0,\mathcal{T}^x_{h}} &\leq 2L\|f - f_h\|_{0,\mathcal{T}_h}, 
				\\
				\|\rho - \rho_h\|_{L^\infty(\mathcal{T}^x_{h})} &\leq 4L^2\|f - f_h\|_{L^\infty(\mathcal{T}_h)},
				\\
				\|\rho V - \rho_h V_h\|_{0,\mathcal{T}^x_{h}} &\leq 4L^2\|f - f_h\|_{0,\mathcal{T}_h}.
			\end{aligned}
		\end{equation*}		
	\end{lem}
The estimate in the above lemma are a mere consequence of the Cauchy-Schwarz inequality. A similar set of estimates were proved for a dG scheme approximating solutions to the Vlasov-viscous Burgers' system. Please consult \cite[Lemma $3.7$]{hutridurga2023discontinuous} for the proofs.	
%

	\section{A priori Error Estimates}\label{err}
	This section discusses some a priori error estimates for the discrete solution.
	\subsection{Error estimates for Stokes' system}
	
	The error equation for the Stokes' problem is
	\begin{equation*}\label{er}
		\begin{aligned}
			\mathcal{A}_{h}\left((\bm{u - u_h},p - p_h),(\bm{\phi_h},q_h)\right) + \left(\rho\bm{u} - \rho_h\bm{u_h}, \bm{\phi_h}\right) = \left(\rho V - \rho_h V_h, \bm{\phi_h}\right).
		\end{aligned}
	\end{equation*}
	To estimate $\bm{u - u_h}$, we rewrite it as
	\begin{equation*}
		\begin{aligned}
			\bm{u} - \bm{u_h} &= \bm{u} - \bm{\widetilde{u}_h} + \bm{\tilde{u}_h} - \bm{u_h}
			= \bm{u} -\bm{\mathcal{P}_xu} + \bm{\mathcal{P}_xu} - \bm{\widetilde{u}_h} + \bm{\widetilde{u}_h} - \bm{u_h}.
		\end{aligned}
	\end{equation*}
	Here, $\bm{\widetilde{u}_h}$ is the solution of an auxiliary problem \eqref{ads} given below. Projection estimate gives the estimate of $\|\bm{u} - \bm{\mathcal{P}_xu}\|_{\bm{L^2}}$. So, to estimate $\|\bm{u} - \bm{u_h}\|_{\bm{L^2}}$, it is sufficient to estimate $\|\bm{\mathcal{P}_xu} - \bm{\widetilde{u}_h}\|_{\bm{L^2}}$ and $\|\bm{\widetilde{u}_h} - \bm{u_h}\|_{\bm{L^2}}$. We will obtain these estimates in Lemmas \ref{tb1} and \ref{tb}, respectively.

	We now consider the following auxiliary discrete Stokes' problem for the unknown $(\bm{\tilde{u}_h},p_h) \in U_h\times P_h$:  
	\begin{equation}\label{ads}
		\mathcal{A}_{h}((\bm{\widetilde{u}_h},\widetilde{p}_h),(\bm{\phi_h},q_h)) + \left(\rho\bm{\widetilde{u}_h}, \bm{\phi_h}\right) = \left(\rho V, \bm{\phi_h}\right) \hspace{5mm} \forall\hspace{1mm}(\bm{\phi_h},q_h) \in U_h\times P_h,
	\end{equation}	
	where, $\mathcal{A}_{h}((\bm{\widetilde{u}_h},\widetilde{p}_h),(\bm{\phi_h},q_h))$ is defined as in \eqref{A_h}. Note that the above auxiliary problem differs from the original discrete generalized Stokes'system \eqref{A_h} in that the discrete local density $\rho_h$ and discrete local momentum $\rho_hV_h$ are replaced by the local density $\rho$ and the local momentum $\rho V$ associated with the continuum solution $f$.


	\begin{lem}\label{exs}
		The auxiliary discrete Stokes' problem \eqref{ads} satisfies the discrete inf-sup condition, i.e. there exists a constant $\alpha > 0$, independent on $h$, such that
		\begin{equation}\label{dic}
			\alpha\vertiii{(\bm{\widetilde{u}_h},\widetilde{p}_h)} \leq \sup_{(\bm{\phi_h},q_h) \in U_h\times P_h\setminus\{(0,0)\}}\frac{\mathcal{A}_h((\bm{\widetilde{u}_h},\widetilde{p}_h),(\bm{\phi_h},q_h)) + \left(\rho\bm{\widetilde{u}_h}, \bm{\phi_h}\right)}{\vertiii{(\bm{\phi_h},q_h)}}.
		\end{equation}
	Furthermore, we have
		\begin{equation}\label{ahbd1}
			|\mathcal{A}_h((\bm{\widetilde{u}_h},\widetilde{p}_h),(\bm{\phi_h},q_h)) + \left(\rho\bm{\widetilde{u}_h, \phi_h}\right)| \lesssim \vertiii{(\bm{\widetilde{u}_h},\widetilde{p}_h)}\vertiii{(\bm{\phi_h},q_h)}.
		\end{equation}
	\end{lem}
	
	\begin{proof}
		Let right hand side of \eqref{dic} be denoted by $\mathcal{L}(\bm{\widetilde{u}_h},\widetilde{p}_h)$. Note that
		\begin{equation}\label{s}
			C\vertiii{\bm{\widetilde{u}_h}}^2 + \sum_{F \in \Gamma_x}h_x\|[\![\widetilde{p}_h]\!]\|^2_{\bm{L^2}} \leq \mathcal{L}(\bm{\widetilde{u}_h},\widetilde{p}_h)\vertiii{(\bm{\widetilde{u}}_h,\widetilde{p}_h)}.
		\end{equation}
		Note that $b_h(\bm{\phi_h},p_h)$ satisfies the discrete inf-sup condition (for detailed proof, see \cite[Lemma 6.10, p. 251]{2}), i.e. 
		\begin{equation*}
			\|\widetilde{p}_h\|_{0,\mathcal{T}_h^x} \lesssim \sup_{\bm{\phi_h} \in U_h \setminus \{0\}}\frac{b_h(\bm{\phi_h},\widetilde{p}_h)}{\vertiii{\bm{\phi_h}}} + \left(\sum_{F \in \Gamma_x}h_x\|[\![\widetilde{p}_h]\!]\|^2_{\bm{L^2}}\right)^\frac{1}{2}.
		\end{equation*}
		Taking $q_h = 0$ in \eqref{A_h}, it follows that
		\begin{equation*}
			b_h(\bm{\phi_h},\widetilde{p}_h) = \mathcal{A}_{h}((\bm{\widetilde{u}_h},\widetilde{p}_h),(\bm{\phi_h},0)) - a_h(\bm{\widetilde{u}_h}, \bm{\phi_h}), 
		\end{equation*}
		and hence, there holds
		\begin{equation*}
			\begin{aligned}
				\|\widetilde{p}_h\|_{0,\mathcal{T}_h^x} \lesssim&\,\, \sup_{\bm{\phi_h} \in U_h\times P_h \setminus \{(0,0)\}}\left(\frac{\mathcal{A}_h((\bm{\widetilde{u}_h},\widetilde{p}_h),(\bm{\phi_h},0))}{\vertiii{(\bm{\phi_h},0)}} - \frac{a_h(\bm{\widetilde{u}_h,\phi_h})}{\vertiii{\bm{\phi_h}}}\right) 
				\\
				&+ \left(\sum_{F \in \Gamma_x}h_x\|[\![\widetilde{p}_h]\!]\|^2_{\bm{L^2}}\right)^\frac{1}{2}
				\\
				\lesssim&\,\, \sup_{\bm{\phi_h}\in U_h \setminus \{0\}}\frac{\vert a_h(\bm{\widetilde{u}_h,\phi_h})\vert}{\vertiii{\bm{\phi_h}}} + \sup_{\bm{\phi_h}\in U_h \setminus \{0\}}\frac{\vert(\rho\bm{\widetilde{u}_h,\phi_h})\vert}{\vertiii{\bm{\phi_h}}}
				\\
				&+ \mathcal{L}(\bm{\widetilde{u}_h},\widetilde{p}_h) + \left(\sum_{F \in \Gamma_x}h_x\|[\![\widetilde{p}_h]\!]\|^2_{\bm{L^2}}\right)^\frac{1}{2}
				\\
				\lesssim&\,\, \vertiii{\bm{\widetilde{u}_h}} + \mathcal{L}(\bm{\widetilde{u}_h},\widetilde{p}_h) + \left(\sum_{F \in \Gamma_x}h_x\|[\![\widetilde{p}_h]\!]\|^2_{\bm{L^2}}\right)^\frac{1}{2}.		
			\end{aligned}
		\end{equation*}
	Squaring on both sides yields
		\begin{equation*}
			\begin{aligned}
				\|\widetilde{p}_h\|^2_{0,\mathcal{T}_h^x} &\lesssim \mathcal{L}^2(\bm{\widetilde{u}_h},\widetilde{p}_h) + \vertiii{\bm{\widetilde{u}_h}}^2 + \sum_{F \in \Gamma_x}h_x\|[\![\widetilde{p}_h]\!]\|^2_{\bm{L^2}}
				\\
				&\lesssim \mathcal{L}^2(\bm{\widetilde{u}_h},\widetilde{p}_h) + \mathcal{L}(\bm{\widetilde{u}_h},\widetilde{p}_h)\vertiii{(\bm{\widetilde{u}_h},\widetilde{p}_h)},
			\end{aligned}
		\end{equation*}
	where we have use \eqref{s}.\\
		This along with \eqref{s} yields
		\begin{equation*}
			\vertiii{(\bm{\widetilde{u}_h},\widetilde{p}_h)}^2 \lesssim \mathcal{L}^2(\bm{\widetilde{u}_h},\widetilde{p}_h) + \mathcal{L}(\bm{\widetilde{u}_h},\widetilde{p}_h)\vertiii{(\bm{\widetilde{u}_h},\widetilde{p}_h)},
		\end{equation*}
	from which \eqref{dic} follows.	A use of \eqref{ahbd} with the fact that $\rho \in L^\infty(\Omega_x)$ (Lemma \ref{lem:density}) shows the boundedness result.
		This completes the proof.
	\end{proof}
	
	Note that the local density $\rho \geq 0$. Thanks to the inf-sup condition and the boundedness property proved in Lemma \ref{exs}, an application of the Lax-Milgram lemma demonstrate the existence of a unique solution $(\bm{\widetilde{u}_h},\widetilde{p}_h) \in U_h \times P_h$ for the auxiliary problem \eqref{ads}.

To estimate $\bm{u - \widetilde{u}_h}$, we use the following identity:
	\begin{equation}\label{eu}
		\mathcal{A}_h\left((\bm{u - \widetilde{u}_h},p - \widetilde{p}_h),(\bm{\phi_h},q_h)\right) + \left(\rho\left(\bm{u - \widetilde{u}_h}\right), \bm{\phi_h}\right) = 0. 
	\end{equation}	
	Using $L^2$-projections, we rewrite
	\begin{equation}\label{thetauetau}
		\begin{aligned}
			\bm{u - \widetilde{u}_h} &= \left(\bm{u - \mathcal{P}_xu}\right) + \left(\bm{\mathcal{P}_xu - \widetilde{u}_h}\right) = \bm{\theta_u - \eta_u},
			\\
			p - \widetilde{p}_h &= \left(p - \mathcal{P}_xp\right) + \left(\mathcal{P}_xp - \widetilde{p}_h\right) = \theta_p - \eta_p,
		\end{aligned}
	\end{equation}
	where,
	\begin{equation*}
		\begin{aligned}
			\bm{\theta_u} = \bm{\mathcal{P}_xu - \widetilde{u}_h}, \quad \bm{\eta_u} = \bm{\mathcal{P}_xu - u}, \quad \theta_p = \mathcal{P}_xp - \widetilde{p}_h,\quad  \eta_p = \mathcal{P}_xp - p.
		\end{aligned}
	\end{equation*}
	Using \eqref{thetauetau}, we rewrite \eqref{eu} as
	\begin{equation}\label{etheta}
		\mathcal{A}_h\left(\left(\bm{\theta_u},\theta_p\right),\left(\bm{\phi_h},q_h\right)\right) + \left(\rho\bm{\theta_u},\bm{\phi_h}\right) = \mathcal{A}_h\left(\left(\bm{\eta_u},\eta_p\right),\left(\bm{\phi_h},q_h\right)\right) + \left(\rho\bm{\eta_u},\bm{\phi_h}\right).
	\end{equation}


	\begin{lem}\label{tb1}
		Let $(\bm{u},p)$ denote the unique solution of the Stokes' equation \eqref{contstokes}. Let $(\bm{\widetilde{u}_h},\widetilde{p}_h) \in U_h \times P_h$ be the solution of the auxiliary problem \eqref{ads}. Assume $(\bm{u},p) \in L^\infty([0,T];\bm{H^{k+1}}) \times L^\infty([0,T];H^k)$. Then, there exists a positive constant $C$, independent of $h$, such that 
		\begin{equation*}
			\|\bm{u - \widetilde{u}_h}\|_{\bm{L^2}} + h_x\vertiii{\bm{u - \widetilde{u}_h}}  + h_x\|p - \widetilde{p}_h\|_{0,\mathcal{T}_{h}^x} \leq Ch^{k+1}_x.
		\end{equation*}
	\end{lem}
	
	\begin{proof}
		Note that
		\[
		\vertiii{\bm{u - \widetilde{u}_h}} \leq \vertiii{\bm{\theta_u}} + \vertiii{\bm{\eta_u}}.
		\]
		Projection estimate yields $\vertiii{\bm{\eta_u}} \leq Ch^k$. Hence, it suffices to estimate $\vertiii{\bm{\theta_u}}$.\\
		Employing the discrete inf-sup condition and the boundedness property(proved in Lemma \ref{exs}) in the equation \eqref{etheta}, we obtain
		\begin{equation*}
			\alpha\vertiii{\left(\bm{\theta_u},\theta_p\right)} \leq \vertiii{\left(\bm{\eta_u},\eta_p\right)} . 
		\end{equation*}
		A use of the projection estimate yields
		\begin{equation}\label{tup}
			\vertiii{\left(\bm{\theta_u},\theta_p\right)} \leq Ch^k_x.
		\end{equation}
		Similarly, as
		\begin{equation}\label{ull2}
			\|\bm{u - \widetilde{u}_h}\|_{\bm{L^2}} \leq \|\bm{\eta_u}\|_{\bm{L^2}} + \|\bm{\theta_u}\|_{\bm{L^2}},
		\end{equation}
		and as we have $\|\bm{\eta_u}\|_{\bm{L^2}} \leq Ch^{k+1}_x$ from the projection estimate, it is enough to estimate $\|\bm{\theta_u}\|_{\bm{L^2}}$.\\
		Let $(\bm{\zeta},\xi) \in [\bm{H^1 \cap L_0^2}] \times L^2 $ denote the solution of the Stokes' problem \eqref{contstokes} with forcing term $\bm{\theta_u} \in \bm{L^2}$.
		Then, owing to Stokes' regularity \cite{giga1991abstract,amrouche1991existence,ladyzhenskaya1969mathematical}, we have
		\begin{equation}\label{reg}
			\|\bm{\zeta}\|_{\bm{H^2}} + \|\xi\|_{H^1} \leq C\|\bm{\theta_u}\|_{\bm{L^2}}.
		\end{equation} 
	Moreover, since $\bm{\zeta} \in \bm{H^2 \cap L^2_0}, \nabla\cdot\bm{\zeta} = 0, [\![\nabla\zeta_i]\!] = 0, i = 1,2$ across all $F \in \Gamma_x^0$ and $[\![\bm{\zeta}]\!] = 0$ across all $F \in \Gamma_x$, we obtain
		\begin{equation*}
			a_h(\bm{\theta_u,\zeta}) = a_h(\bm{\zeta,\theta_u}) = -\left(\Delta_x\bm{\zeta}, \bm{\theta_u}\right)_{\Omega_x},
		\end{equation*}
		\begin{equation*}
			\begin{aligned}
				b_h(\bm{\zeta},\theta_p) = 0.
			\end{aligned}
		\end{equation*}
		Since $\xi \in H^1(\Omega_x), [\![\xi]\!]\cdot\bm{n} = 0$ across all $F \in \Gamma_x^0$, we obtain
		\begin{equation*}
			\begin{aligned}		
				b_h(\bm{\theta_u},\xi) &= \left(\bm{\theta_u},\nabla_x\xi\right)_{\Omega_x},
				\\
				s_h(\theta_p,\xi) &= 0.
			\end{aligned}
		\end{equation*}
		Hence,
		\begin{equation*}
			\begin{aligned}
				\|\bm{\theta_u}\|^2_{\bm{L^2}} &= \left(\left(-\Delta_x\bm{\zeta} + \nabla_x\xi + \rho\bm{\zeta}\right),\bm{\theta_u}\right)_{\Omega_x}
				\\
				&= \mathcal{A}_h((\bm{\theta_u},\theta_p),(\bm{\zeta},\xi)) + \left(\rho\bm{\theta_u, \zeta}\right).
			\end{aligned}
		\end{equation*}
		Note that $\mathcal{A}_h((\bm{\theta_u}, \theta_p),(\bm{\mathcal{P}_x\zeta},\mathcal{P}_xp)) + \left(\rho\bm{\theta_u, \mathcal{P}_x\zeta}\right) = 0$.\\
		Using the boundedness property and the projection estimates, we obtain
		\begin{equation*}
			\begin{aligned}
				\|\bm{\theta_u}\|^2_{\bm{L^2}} &= \mathcal{A}_h((\bm{\theta_u},\theta_p),(\bm{\zeta - \mathcal{P}_x\zeta}, \xi - \mathcal{P}_x\xi)) + \left(\rho\bm{\theta_u, \zeta - \mathcal{P}_x\zeta}\right)
				\\
				&\lesssim \vertiii{(\bm{\theta_u},\theta_p)}\vertiii{(\bm{\zeta - \mathcal{P}_x\zeta}, \xi - \mathcal{P}_x\xi)} 
				\\
				&\lesssim h_x\vertiii{(\bm{\theta_u},\theta_p)}\left(\|\bm{\zeta}\|_{\bm{H^2}} + \|\xi\|_{H^1}\right)
				\\
				&\lesssim h_x\vertiii{(\bm{\theta_u},\theta_p)}\|\bm{\theta_u}\|_{\bm{L^2}},
			\end{aligned}
		\end{equation*}
	where we have used \eqref{reg} in the final step.
	Hence we get
		\begin{equation*}
			\|\bm{u - \widetilde{u}_h}\|_{\bm{L^2}} \leq Ch^{k+1}_x.
		\end{equation*}   	
		Observe that 
		\[
		\|p - \widetilde{p}_h\|_{0,\mathcal{T}_{h}^x} \leq \|\theta_p\|_{0,\mathcal{T}_{h}^x} + \|\eta_p\|_{0,\mathcal{T}_{h}^x}.
		\]
		From the projection estimate, we have $\|\eta_p\|_{0,\mathcal{T}_{h}^x} \leq Ch^k_x$ and from \eqref{tup} $\|\theta_p\|_{0,\mathcal{T}_{h}^x} \leq Ch^k_x$.
		This completes the proof.
	\end{proof}


	Now to find the estimate on the term $\bm{\widetilde{u}_h - u_h}$, we use	
	\begin{equation}\label{ahti}
		\mathcal{A}_{h}\left((\bm{\widetilde{u}_h - u_h}, \widetilde{p}_h - p_h),(\bm{\phi_h},q_h)\right) + \left(\rho\bm{\widetilde{u}_h} - \rho_h\bm{u_h}, \bm{\phi_h}\right)  = \left(\rho V - \rho_h V_h, \bm{\phi_h}\right).
	\end{equation}
	
	For the rest of this paper, we assume that there are positive constants $K, h_0$ and $K^*$, with $K^* \geq 2K$ such that
	\begin{equation*}
		\|\bm{u}\|_{L^\infty([0,T];\bm{L^\infty})} \leq K,
	\end{equation*}
	and
	\begin{equation}\label{uh}
		\|\bm{u_h}\|_{L^\infty([0,T];\bm{L^\infty})} \leq K^*, \quad \mbox{for}\,\,\,0 < h \leq h_0.
	\end{equation}

	\begin{lem}\label{tb}
		Let $(\bm{\widetilde{u}_h},\widetilde{p}_h) \in U_h \times P_h$ be solution of the auxiliary problem \eqref{ads} and let $(\bm{u_h},p_h) \in U_h \times P_h$ be the solution to the discrete generalized Stokes' problem \eqref{ch}. Then,
		\begin{equation}\label{eq:ul222}
			\vertiii{\bm{\widetilde{u}_h - u_h}} + \|\widetilde{p}_h - p_h\|_{0,\mathcal{T}_h^x} \leq C(K^*)\|f - f_h\|_{0,\mathcal{T}_h}.
		\end{equation}		
	\end{lem}
	
	\begin{proof}
		The equation \eqref{ahti} can be rewritten as	
		\begin{equation*}
			\begin{aligned}
				\mathcal{A}_h&\left((\bm{\widetilde{u}_h - u_h}, \widetilde{p}_h - p_h),(\bm{\phi_h},q_h)\right) + \left(\rho\left(\bm{\widetilde{u}_h} - \bm{u_h}\right), \bm{\phi_h}\right) 
				\\
				&= \left(\left(\rho_h - \rho\right)\bm{u_h}, \bm{\phi_h}\right) + \left(\rho V - \rho_h V_h, \bm{\phi_h}\right).
			\end{aligned}
		\end{equation*}
	Then, \eqref{eq:ul222} is a consequence of the discrete inf-sup condition and the boundedness property proved in Lemma \ref{exs} and the $\|\cdot\|_{0,\mathcal{T}_h}$ bounds of $\left(\rho - \rho_h\right)$ and $\left(\rho V - \rho_hV_h\right)$ from Lemma \ref{rhoh2}. The dependence of the constant in \eqref{eq:ul222} on $K^*$ is due to our assumption \eqref{uh}.
	\end{proof}

	Putting together the estimate of Lemma \ref{tb1} and Lemma \ref{tb} gives the following estimate. 
	
	\begin{thm}\label{ul2estimate}
		Let $(\bm{u},p)$ denote the unique solution of the Stokes' equation \eqref{contstokes} and let $(\bm{u_h},p_h) \in U_h\times P_h$ solve the discrete generalized Stokes' problem \eqref{ch}. Assume that $(\bm{u},p) \in L^\infty([0,T];\bm{H^{k+1}}) \times L^\infty([0,T];H^k(\Omega_x))$. Then, there exists a positive constant $C$ independent of $h$, such that
		\begin{equation*}
			\begin{aligned}
				\|\bm{u-u_h}\|_{\bm{L^2}} + h_x\vertiii{\bm{u - u_h}} + h_x\|p - p_h\|_{0,\mathcal{T}_h^x} \leq Ch^{k+1}_x + C(K^*_0)\|f - f_h\|_{0,\mathcal{T}_h}.
			\end{aligned}
		\end{equation*} 
	\end{thm}
	

	\subsection{Error estimates for Vlasov equation }

The error equation for the Vlasov equation is
\begin{equation}
	\begin{aligned}
		\left(\partial_t\left(f - f_h\right),\psi_h\right) + \sum_{R\in\mathcal{T}_h}\left(\mathcal{B}_{h,R}(\bm{u};f,\psi_h) - \mathcal{B}_{h,R}(\bm{u_h};f_h,\psi_h)\right) = 0 \quad \forall \quad \psi \in \mathcal{Z}_h
	\end{aligned}
\end{equation}
	
where $\mathcal{B}_{h,R}(\cdot;\cdot,\cdot)$ is defined in \eqref{bh}.

	Setting $e_f = f - f_h$, we rewrite the above equation as  
	\begin{equation}\label{error}
		\begin{aligned}
			\left(\partial_t e_f,\psi_h\right) + a_h^0(e_f,\psi_h) + \mathcal{N}(\bm{u};f,\psi_h) - \mathcal{N}^h(\bm{u_h};f_h,\psi_h) = 0
			\quad\forall\, \psi_h \in \mathcal{Z}_h , 
		\end{aligned}
	\end{equation}
	where
	\begin{equation}\label{ae}
		a_h^0(e_f,\psi_h) = -\sum_{R \in \mathcal{T}_{h}}\int_{R}e_f v\cdot \nabla_x\psi_h\,{\rm d}v\,{\rm d}x
		+ \sum_{T^v \in \mathcal{T}_{h}^v}\int_{T^v}\int_{\Gamma_x}\{v e_f\}_\alpha\cdot[\![\psi_h]\!]\,{\rm d}s(x)\,{\rm d}v,
	\end{equation}	
	\begin{equation}\label{N}
		\mathcal{N}(\bm{u};f,\psi_h) = -\sum_{R \in \mathcal{T}_{h}}\int_{R} f (\bm{u} - v)\cdot\nabla_v\psi_h\,{\rm d}v\,{\rm d}x \,\, + \sum_{T^x \in \mathcal{T}_{h}^x}\int_{T^x}\int_{\Gamma_v}f(\bm{u} - v)\cdot [\![\psi_h]\!]\,{\rm d}s(v)\,{\rm d}x \, ,
	\end{equation}
	and
	\begin{equation}\label{Nh}
		\begin{aligned}
			\mathcal{N}^h(\bm{u_h};f_h,\psi_h) = &-\sum_{R \in \mathcal{T}_{h}}\int_{R} f_h (\bm{u_h} - v)\cdot\nabla_v\psi_h\,{\rm d}v\,{\rm d}x 
			\\
			&+ \sum_{T^x \in \mathcal{T}_{h}^x}\int_{T^x}\int_{\Gamma_v}\{(\bm{u_h} - v) f_h\}_\beta\cdot[\![\psi_h]\!]\,{\rm d}s(v)\,{\rm d}x \, .
		\end{aligned}
	\end{equation}

	\subsection{Special Projection:}
	
	To define the projection operator, first we define one dimensional projection operator $\pi^{\pm}$ whose definition is based on the projection defined in \cite[Section 3.1]{schotzau2000time}. Let $\pi^+ : H^{\frac{1}{2}+\epsilon} \rightarrow X_h$ be the projection operator defined by
	\begin{equation}\label{pi+}
		\int_{T^x}(\pi^+(w) - w)q_h\,{\rm d}x = 0 \quad \forall\,\, q_h \in \mathbb{P}^{k-1}(T^x),
	\end{equation}
	together with the matching condition:
	\begin{equation*}\label{pii}
		\pi^+(w(x^+)) = w(x^+),
	\end{equation*}
and let $\pi^- : H^{\frac{1}{2}+\epsilon} \rightarrow X_h$ be the projection operator defined as
\begin{equation}\label{pi-}
	\int_{T^x}(\pi^+(w) - w)q_h\,{\rm d}x = 0 \quad \forall\,\, q_h \in \mathbb{P}^{k-1}(T^x),
\end{equation}
together with the matching condition:
\begin{equation*}\label{pii}
	 \pi^-(w(x^-)) = w(x^-).
\end{equation*}
The projection estimate for $\pi^\pm$ are available(refer to \cite{schotzau2000time}) 
	\begin{equation*}
		\|w - \pi^\pm(w)\|_{0,T^x} \leq Ch^{k+1}|w|_{k+1,T^x},
	\end{equation*}
	where $C$ is a constant depending only on the shape-regularity of the mesh and the polynomial degree.\\
	Inspired by the projection operator defined in \cite{ayuso2009discontinuous,de2012discontinuous}, we define projection operator $\Pi_h : \mathcal{C}^0(\Omega) \rightarrow \mathcal{Z}_h$ as follows: Let $R = T^x \times T^v$ be an arbitrary element of $\mathcal{T}_h$ and let $w \in \mathcal{C}^0({R})$. The restriction of $\Pi_h(w)$ to $R$ is defined by 
	\begin{equation}\label{pipi}
		\Pi_h(w) =
		\left\{
		\begin{aligned}
			&(\tilde{\Pi}_x\otimes\tilde{\Pi}_v)(w)\quad\mbox{if sign}\left(\left(\bm{u} - v\right)\cdot \bm{n}\right) = \mbox{constant},
			\\
			&(\tilde{\Pi}_x \otimes \tilde{\mathcal{P}}_v)(w) \quad\mbox{if sign}\left(\left(\bm{u} - v\right)\cdot \bm{n}\right) \neq \mbox{constant},
		\end{aligned}
		\right.
	\end{equation}
	where $\tilde{\Pi}_x : \mathcal{C}^0(\Omega_x) \rightarrow X_h$ and $\tilde{\Pi}_v : \mathcal{C}^0(\Omega_v) \rightarrow V_h$ are the $2$-dimensional projection operators as follows: (similar to \cite{lesaint1974finite,4})
	Taking $v = [v_1, v_2]^t$ and $u = [u_1, u_2]^t$,
	\begin{equation*}\label{pix}
		\tilde{\Pi}_x(w) = 
		\left\{
		\begin{aligned}
			&\pi_{x,1}^- \otimes \pi_{x,2}^- \quad\mbox{if}\quad v_1 > 0, v_2 > 0,
			\\
			&\pi_{x,1}^+ \otimes \pi_{x,2}^- \quad\mbox{if}\quad v_1 < 0, v_2 > 0,
			\\
			&\pi_{x,1}^+ \otimes \pi_{x,2}^+ \quad\mbox{if}\quad v_1 < 0, v_2 < 0,
			\\
			&\pi_{x,1}^- \otimes \pi_{x,2}^+ \quad\mbox{if}\quad v_1 > 0, v_2 < 0,
		\end{aligned}
		\right.
	\end{equation*}
	and
	\begin{equation*}\label{piv}
		\tilde{\Pi}_v(w) = 
		\left\{
		\begin{aligned}
			&\pi_{v,1}^- \otimes \pi_{v,2}^- \quad\mbox{if}\quad u_1 -v_1 > 0, u_2 - v_2 > 0,
			\\
			&\pi_{v,1}^+ \otimes \pi_{v,2}^- \quad\mbox{if}\quad u_1 -v_1 < 0, u_2 - v_2 > 0,
			\\
			&\pi_{v,1}^- \otimes \pi_{v,2}^+ \quad\mbox{if}\quad u_1 -v_1 > 0, u_2 - v_2 < 0,
			\\
			&\pi_{v,1}^+ \otimes \pi_{v,2}^+ \quad\mbox{if}\quad u_1 -v_1 < 0, u_2 - v_2 < 0,
		\end{aligned}
		\right.
	\end{equation*}
	where the subscript $i$ in $\pi^\pm_{r,i} (r = x \hspace{1mm}\mbox{or}\hspace{1mm} v)$ refers the fact that projection is along the $i$th component in the $r$ space.\\
	In \eqref{pipi}, $\tilde{\mathcal{P}}_v: L^2(\Omega_v) \rightarrow V_h$, which accounts for the cases where $(\bm{u} - v)\cdot \bm{n}$ changes sign across any single element $T^x \times \partial T^v$ is given by
	\begin{equation*}
		\tilde{\mathcal{P}}_v(w) = 
		\left\{
		\begin{aligned}
			&[\mathcal{P}_{v,1}\otimes\pi_{v,2}^\pm](w) \quad\mbox{if sign}(u_1 - v_1) \neq \hspace{1mm}\mbox{constant and sign}(u_2 - v_2) = \mbox{constant},
			\\
			&[\pi_{v,1}^\pm \otimes \mathcal{P}_{v,2}](w) \quad\mbox{if sign}(u_1 - v_1) = \mbox{constant and sign}(u_2 - v_2) \neq \mbox{constant},
			\\
			&[\mathcal{P}_{v,1}\otimes\mathcal{P}_{v,2}](w)\quad\mbox{if sign}(u_1 - v_1) \neq \mbox{constant and sign}(u_2 - v_2) \neq \mbox{constant},
		\end{aligned}
		\right.
	\end{equation*}
	where, $\mathcal{P}_{v,i}, i = 1,2,$ stands for the standard one-dimensional projection along the $v_i$ direction.

	The following Lemma deals with the approximation properties of projection operator $\Pi_h$, whose proof is similar to \cite[Lemma $4.1$, p. $19$]{de2012discontinuous}. 
	\begin{lem}\label{westimate}
		Let $w \in H^{k+1}(R), s \geq 0$ and let $\Pi_h$ be the projection operator defined through \eqref{pi+}-\eqref{pipi}. Then, for all $e \in (\partial T^x \times T^v) \cup (T^x \times \partial T^v)$, there holds
		\begin{equation}\label{w}
			\|w - \Pi_h(w)\|_{0,R} + h^\frac{1}{2}\|w - \Pi_h(w)\|_{0,e} \leq Ch^{k+1}\|w\|_{k+1,R}.
		\end{equation}
	\end{lem}
	
	Summing estimates \eqref{w} from Lemma \ref{westimate}, over elements of the partition $\mathcal{T}_h$, we obtain
	\begin{equation}\label{wprojection}
		\begin{aligned}
			\|w - \Pi_h(w)\|_{0,\mathcal{T}_h} &+ h^{\frac{1}{2}}\|w - \Pi_h(w)\|_{0,\Gamma_x\times \mathcal{T}^v_{h}}
			\\
			& + h^{\frac{1}{2}}\|w - \Pi_h(w)\|_{0,\mathcal{T}^x_{h}\times\Gamma_v} \leq Ch^{k+1}\|w\|_{k+1,\Omega}.
		\end{aligned}
	\end{equation}
	

	Using the special projection, split $f - f_h$ as
	\begin{equation}\label{p}
		e_f := f - f_h := \left(\Pi_h(f) - f_h\right) - \left(\Pi_h(f) - f\right) := \theta_f - \eta_f,
	\end{equation}
	where
	\begin{equation*}
		\theta_f = \Pi_h(f) - f_h \quad \mbox{and} \quad \eta_f = \Pi_h(f) - f.
	\end{equation*}

	\begin{lem}\label{lemn}
		Let $\bm{u} \in C^0(\Omega_x), f \in C^0(\Omega)$ and $f_h \in \mathcal{Z}_h$ with $k \geq 0$. Then, the following identity holds  
		\begin{align*}
			\mathcal{N}(\bm{u};f,\theta_f) - \mathcal{N}^h(\bm{u_h};f_h,\theta_f) &= \sum_{T^x \in \mathcal{T}_{h}^x}\int_{T^x}\int_{\Gamma_v}\frac{|(\bm{u_h} - v)\cdot \bm{n}|}{2}\vert[\![\theta_f]\!]\vert^2\,{\rm d}s(v)\,{\rm d}x
			\\
			& + \sum_{R \in \mathcal{T}_{h}}\int_{R} \left((\bm{u-u_h})\cdot\nabla_v f\,\theta_f - \theta^2_f\right)\,{\rm d}v\,{\rm d}x + \mathcal{K}^2(\bm{u_h} - v,f,\theta_f),
		\end{align*}
		where
		\begin{equation}\label{K2}
			\mathcal{K}^2(\bm{u_h} - v,f,\theta_f) = \sum_{R \in \mathcal{T}_{h}}\int_R \eta_f (\bm{u_h} - v)\cdot \nabla_v \theta_f\,{\rm d}v\,{\rm d}x - \sum_{T^x \in \mathcal{T}_{h}^x}\int_{T^x}\int_{\Gamma_v}\{(\bm{u_h} - v)\eta_f\}_\beta\cdot[\![\theta_f]\!]\,{\rm d}s(v)\,{\rm d}x\, .
		\end{equation}
	\end{lem}
	The proof of the above Lemma is similar to the proof of Lemma $4.5$ in \cite{de2012discontinuous} by replacing $E = \bm{u} - v$ and $E_h = \bm{u_h} - v$. 
	
	
	After choosing $\psi_h = \theta_f$ in error equation \eqref{error} and a use of \eqref{p} and Lemma \ref{lemn}, we can rewrite our error equation as
	\begin{equation}\label{err1}
		\begin{aligned}
			&\left(\partial_t\theta_f, \theta_f\right) + \sum_{T^v \in \mathcal{T}_{h}^v}\int_{T^v}\int_{\Gamma_x}\frac{|v\cdot \bm{n}|}{2}\vert[\![\theta_f]\!]\vert^2\,{\rm d}s(x)\,{\rm d}v  + \sum_{T^x \in \mathcal{T}_{h}^x}\int_{T^x}\int_{\Gamma_v}\frac{|(\bm{u_h} - v)\cdot \bm{n}|}{2}\vert[\![\theta_f]\!]\vert^2\,{\rm d}s(v)\,{\rm d}x
			\\
			& = \left(\partial_t\eta_f, \theta_f\right) - \mathcal{K}^1(v,\eta_f,\theta_f) - \sum_{R \in \mathcal{T}_{h}}\int_{R} \left((\bm{u-u_h})\cdot\nabla_v f\,\theta_f - \theta^2_f\right)\,{\rm d}v\,{\rm d}x - \mathcal{K}^2(\bm{u_h} - v,f,\theta_f) ,
		\end{aligned}
	\end{equation}	
	where,
	\begin{equation}\label{K1}
		\mathcal{K}^1(v,\eta_f,\theta_f) = \sum_{R \in \mathcal{T}_h}\int_{R} \eta_f v\cdot\nabla_x \theta_f\,{\rm d}x\,{\rm d}v - \sum_{T^v \in \mathcal{T}_{h}^v}\int_{T^v}\int_{\Gamma_x}\{v \eta_f\}_\alpha\cdot[\![\theta_f]\!]\,{\rm d}s(x)\,{\rm d}v\, .
	\end{equation}

	
	To show the estimate on $e_f := \theta_f - \eta_f$, it is enough to show the estimate on $\theta_f$ since from Lemma \ref{westimate} estimate on $\eta_f$ are known.\\
	
	Let $\partial T^x = e_1^\pm \cup e_2^\pm$ with $e_i^\pm$ denoting the edges of $\partial T^x$ in the $x_i$-direction and
	\begin{equation*}\label{e}
		e_i^+ = \{e \subset \partial T^x : v\cdot \bm{n} > 0\} \qquad e_i^- = \{e \subset \partial T^x : v\cdot \bm{n} < 0\} \qquad i = 1,2. 
	\end{equation*}	
	Similarly $\partial T^v = \gamma_1^\pm \cup \gamma_2^\pm$ with $\gamma_i^\pm$ denoting the edges of $\partial T^v$ in the $v_i$-direction and
	\begin{equation*}\label{g}
		\gamma_i^+ = \{\gamma \subset \partial T^v : (\bm{u_h} - v)\cdot \bm{n} > 0\} \qquad \gamma_i^- = \{\gamma \subset \partial T^v : (\bm{u_h} - v)\cdot \bm{n} < 0\} \qquad i = 1,2.
	\end{equation*}

	\begin{lem}\label{K1estimate}
		Let $k \geq 1$ and let $f \in C^0([0,T]; W^{1,\infty}(\Omega) \cap H^{k+2}(\Omega))$ be the distribution function solution of \eqref{eq:continuous-model}. Let $f_h(t) \in \mathcal{Z}_h$ be its approximation satisfying \eqref{bh} and let $\mathcal{K}^1(v,\eta_f,\theta_f)$ be defined as in \eqref{K1}. Assume that the partition $\mathcal{T}_h$ is constructed so that none of the components of $v$ vanish inside any element. Then, the following estimates holds:
		\begin{equation}\label{k1estimate}
			|\mathcal{K}^1(v,\eta_f,\theta_f)| \leq Ch^{k+1}\left(\|f\|_{k+1,\Omega} + L\|f\|_{k+2,\Omega}\right)\|\theta_f\|_{0,\mathcal{T}_h},
		\end{equation} 
		for all $t \in [0,T]$.
	\end{lem}
	
	The proof of the above Lemma is similar to the proof of Lemma $4.2$ in \cite{de2012discontinuous}. 
	

	\begin{lem}\label{K2estimate}
		Let $\mathcal{T}_h$ be a Cartesian mesh of $\Omega$, $k \geq 1$ and let $(\bm{u_h},f_h) \in U_h\times \mathcal{Z}_h$ be the solution to \eqref{bh}. Let $(\bm{u},f) \in L^\infty([0,T];\,\bm{W^{1,\infty}\cap H^{k+1}})\times L^\infty([0,T];\,W^{1,\infty}(\Omega)\cap H^{k+2}(\Omega))$ and let $\mathcal{K}^2$ is defined as in \eqref{K2}. Then, the following estimate holds
		\begin{equation}\label{k2estimate}
			\begin{aligned}
				|\mathcal{K}^2(\bm{u_h} - v,f,\theta_f)| \leq& Ch^k\|\bm{u_h - u}\|_{\bm{L^\infty}}\|f\|_{k+1,\Omega}\|\theta_f\|_{0,\mathcal{T}_h} 
				\\
				&+ Ch^{k+1}\left(\|\bm{u}\|_{\bm{W^{1,\infty}}}\|f\|_{k+1,\Omega} + \left(L + \|\bm{u}\|_{\bm{L^\infty}}\right)\|f\|_{k+2,\Omega}\right)\|\theta_f\|_{0,\mathcal{T}_h}.
			\end{aligned}
		\end{equation}		
	\end{lem}
	
	The proof of the above Lemma is similar to the proof of Lemma $4.3$ in \cite{de2012discontinuous} by replacing $E = \bm{u} - v$ and $E_h = \bm{u_h} - v$.


%
%
%
	\subsection{Optimal error estimates:}

	\begin{thm}\label{fl2}
		Let $k \geq 1$ and let $f \in \mathcal{C}^1([0,T];H^{k+2}(\Omega)\cap W^{1,\infty}(\Omega))$ be the solution of the Vlasov-Stokes' problem \eqref{eq:continuous-model}-\eqref{contstokes} with $\bm{u} \in \mathcal{C}^0([0,T];\bm{H^{k+1}}\cap \bm{W^{1,\infty}})$. Further, let $(\bm{u_h},f_h) \in U_h \times \mathcal{Z}_h $ be the dG-dG approximation obtained by solving \eqref{ch} and \eqref{bh}. Then, 
		\[
		\|f(t) - f_h(t)\|_{0,\mathcal{T}_h} \leq C(K^*) h^{k+1} \quad\forall\, t \in [0,T]
		\]
		where $C(K^*)$ depends on the final time $T$, the polynomial degree $k$, the shape regularity of the partition and the continuous solution $(\bm{u},f)$. 
	\end{thm}
	
	\begin{proof}
		Note that
		\begin{equation*}
			\|f - f_h\|_{0,\mathcal{T}_h} \leq \|\eta_f\|_{0,\mathcal{T}_h} + \|\theta_f\|_{0,\mathcal{T}_h}.
		\end{equation*}
		Since from equation \eqref{wprojection}, the estimate of $\|\eta_f\|_{0,\mathcal{T}_h}$ is known, it is enough to estimate $\|\theta_f\|_{0,\mathcal{T}_h}$.\\
		It follows from \eqref{err1} that
		\begin{equation}\label{estimate}
			\begin{aligned}
				\frac{1}{2}\frac{{\rm d}}{{\rm d}t}\|\theta_f\|^2_{0,\mathcal{T}_h} &+ \frac{1}{2}\||v\cdot \bm{n}|^\frac{1}{2}[\![\theta_f]\!]\|^2_{0,\Gamma_x\times \mathcal{T}^v_{h}} + \frac{1}{2}\||(\bm{u_h} - v)\cdot \bm{n}|^\frac{1}{2}[\![\theta_f]\!]\|^2_{0,\mathcal{T}^x_{h}\times \Gamma_v}
				\\
				&= \left(\partial_t\eta_f, \theta_f\right) - \sum_{R \in \mathcal{T}_{h}}\int_R(\bm{u-u_h})\cdot \nabla_v f \theta_f\,{\rm d}v\,{\rm d}x - \mathcal{K}^1(v,\eta_f,\theta_f) 
				\\
				&\quad - \mathcal{K}^2(\bm{u_h} - v,f,\theta_f) + \sum_{R \in \mathcal{T}_{h}}\int_R\theta^2_f\,{\rm d}v\,{\rm d}x
				\\
				&=: I_1 - I_2 - \mathcal{K}^1 - \mathcal{K}^2 + I_3,
			\end{aligned}
		\end{equation}
		where
		\[
		I_1 := \left(\partial_t\eta_f, \theta_f\right), \quad I_2 := \sum_{R \in \mathcal{T}_{h}}\int_R(\bm{u-u_h})\cdot \nabla_v f \theta_f\,{\rm d}v\,{\rm d}x \quad \mbox{and}\quad		
		I_3 := \sum_{R \in \mathcal{T}_{h}}\int_R\theta^2_f\,{\rm d}v\,{\rm d}x.
		\]
		Note that
		\begin{equation}\label{I1estimate}
			\begin{aligned}
				|I_1| &\leq \|\partial_t\eta_f\|_{0,\mathcal{T}_h}\|\theta_f\|_{0,\mathcal{T}_h} \\
				&\leq Ch^{k+1}\|\partial_tf\|_{k+1,\mathcal{T}_h}\|\theta_f\|_{0,\mathcal{T}_h}
				\leq Ch^{2k+2}\|\partial_tf\|^2_{k+1,\mathcal{T}_h} + C\|\theta_f\|^2_{0,\mathcal{T}_h},
			\end{aligned}
		\end{equation}
	where we have used the projection estimate \eqref{wprojection} in the second step. Note that it is here that we require $f$ to be $\mathcal{C}^1$ in the time variable.\\ 
		Observe that 
		\begin{equation}\label{I2estimate}
			\begin{aligned}
				|I_2| & = |\sum_{R \in \mathcal{T}_{h}}\int_R(\bm{u-u_h})\cdot \nabla_v f \theta_f\,{\rm d}v\,{\rm d}x|
				\\
				&\leq \|\bm{u - u_h}\|_{\bm{L^2}}\|\nabla_v f\|_{L^\infty(\Omega)}\|\theta_f\|_{0,\mathcal{T}_h}
				\\
				&\leq \|\bm{u - u_h}\|^2_{\bm{L^2}}\|\nabla_v f\|_{L^\infty(\Omega)} + \|\nabla_v f\|_{L^\infty(\Omega)}\|\theta_f\|^2_{0,\mathcal{T}_h}
				\\
				&\leq C(K^*)h^{2k+2}\|f\|_{W^{1,\infty}(\Omega)}\|f\|_{k+1,\Omega}^2 + C_0(K^*)\|\theta_f\|^2_{0,\mathcal{T}_h},
			\end{aligned}
		\end{equation}
	where we have used the estimate from the Theorem \ref{ul2estimate}. Note that it is here in \eqref{I2estimate} that we require 
	$f$ to be in $W^{1,\infty}$ in the $v$ variable.\\
	The estimate for $\mathcal{K}^1$ in Lemma \ref{K1estimate} yields
		\begin{equation}\label{K1Estimate}
			|\mathcal{K}^1| \leq Ch^{2k+2}\left(\|f\|_{k+1,\Omega} + L\|f\|_{k+2,\Omega}\right)^2 + C\|\theta_f\|^2_{0,\mathcal{T}_h}.
		\end{equation}
		To deal with $\mathcal{K}^2$, we observe that the bound \eqref{k2estimate} in Lemma \ref{K2estimate} yields
		\begin{equation}\label{K2Estimate}
			\begin{aligned}
				|\mathcal{K}^2| &\leq Ch^k\|\bm{u - u_h}\|_{\bm{L^\infty}}\|f\|_{k+1,\Omega}\|\theta_f\|_{0,\mathcal{T}_h} 
				\\
				&+ Ch^{k+1}\left(\|\bm{u}\|_{\bm{W^{1,\infty}}}\|f\|_{k+1,\Omega} + \left(L + \|\bm{u}\|_{\bm{L^\infty}}\right)\|f\|_{k+2,\Omega}\right)\|\theta_f\|_{0,\mathcal{T}_h}
				\\
				&\leq Ch^k\left(h\|\bm{u}\|_{\bm{W^{1,\infty}}} + h^k\|\bm{u}\|_{\bm{H^{k+1}}} + h^{-1}\|\bm{u - u_h}\|_{\bm{L^2}}\right)\|f\|_{k+1,\Omega}\|\theta_f\|_{0,\mathcal{T}_h} 
				\\
				&+ Ch^{k+1}\left(\|\bm{u}\|_{\bm{W^{1,\infty}}}\|f\|_{k+1,\Omega} + \left(L + \|\bm{u}\|_{\bm{L^\infty}}\right)\|f\|_{k+2,\Omega}\right)\|\theta_f\|_{0,\mathcal{T}_h}
				\\
				&\leq Ch^k\left(h\|\bm{u}\|_{\bm{W^{1,\infty}}} + h^k\|\bm{u}\|_{\bm{H^{k+1}}} + h^kC(K^*)\right)\|f\|_{k+1,\Omega}\|\theta_f\|_{0,\mathcal{T}_h} 
				\\
				&+ Ch^{k+1}\left(\|\bm{u}\|_{\bm{W^{1,\infty}}}\|f\|_{k+1,\Omega} + \left(L + \|\bm{u}\|_{\bm{L^\infty}}\right)\|f\|_{k+2,\Omega}\right)\|\theta_f\|_{0,\mathcal{T}_h}
				\\
				&+ C(K^*)\|\theta_f\|^2_{0,\mathcal{T}_h}
				\\
				&\leq C_1(K^*)h^{2k+2} + C_2(K^*)\|\theta_f\|^2_{0,\mathcal{T}_h}.
			\end{aligned}
		\end{equation}
		Here, in the second step, we have used Lemma \ref{L:uinf} which require $\bm{u} \in \bm{W^{1,\infty}}$. In the third step, we employ Theorem \ref{ul2estimate} to estimate $\|\bm{u - u_h}\|_{\bm{L^2}}$.
		
		Substituting the estimates \eqref{I1estimate}-\eqref{K2Estimate} into \eqref{estimate} and using the fact that the last two terms on the left hand side are non-negative, we obtain
		\begin{equation}\label{gron}
			\begin{aligned}
				\frac{{\rm d}}{{\rm d}t}\|\theta_f\|^2_{0,\mathcal{T}_h} \leq C_3(K^*)h^{2k+2} + C_4(K^*)\|\theta_f\|^2_{0,\mathcal{T}_h}.
			\end{aligned}
		\end{equation}
		A standard application of the Gr\"onwall's inequality shows
		\[
		\|\theta_f(t)\|^2_{0,\mathcal{T}_h} \leq C(K^*) h^{2k+2},
		\]
		where $C(K^*)$ is now independent of $h$ and $f_h$, and depends on $t$ and on the solution $(\bm{u},f)$ through its norm. This completes the proof.	
	\end{proof}


	\begin{thm}\label{fi}
		Let $k \geq 1$ and let $(\bm{u},p,f) \in \mathcal{C}^0([0,T];\bm{H^{k+1}\cap W^{1,\infty}}) \times \mathcal{C}^0([0,T];\\H^k(\Omega_x)) \times \mathcal{C}^1([0,T];H^{k+2}(\Omega)\cap W^{1,\infty}(\Omega)) $ be the solution of the Vlasov-Stokes' equation \eqref{eq:continuous-model}-\eqref{contstokes}. Further, let $(\bm{u_h},p_h,f_h) \in \mathcal{C}^0([0,T];U_h) \times \mathcal{C}^0([0,T];P_h) \times \mathcal{C}^1([0,T];\mathcal{Z}_h)$ be the dG-dG approximation solution obtained by solving \eqref{bh}-\eqref{ch}. Then, there holds 
		\[
		\|f(t) - f_h(t)\|_{0,\mathcal{T}_h} + \|(\bm{u - u_h})(t)\|_{\bm{L^2}} + h\|p - p_h\|_{0,\mathcal{T}_h^x} \leq Ch^{k+1} \quad \forall\,\,t \in [0,T]
		\] 
		where $C$ depends on the final time $T$, the polynomial degree $k$, the shape regularity of the partition and the norm of $(\bm{u},f)$.
	\end{thm}
	
	\begin{proof}
		As a consequence of Theorems \ref{ul2estimate} and \ref{fl2}, we arrive
		\begin{equation*}
			\|(\bm{u - u_h})(t)\|_{\bm{L^2}} \leq C(K^*)h^{k+1}, \qquad
			h\|p - p_h\|_{0,\mathcal{T}_h^x} \leq Ch^{k+1},
		\end{equation*}
		and
		\begin{equation*}
			\|(f - f_h)(t)\|_{0,\mathcal{T}_h} \leq C(K^*)h^{k+1}.
		\end{equation*}
		Now to complete the proof we need to show $\bm{u_h}$ is indeed bounded. More precisely, we need to show \eqref{uh}.
		Note that, for small $h$ and $k \geq 1$ from Lemma \ref{L:uinf}, we obtain
		\begin{equation*}
			\begin{aligned}
				\|\bm{u_h}\|_{L^\infty([0,T];\bm{L^\infty})} &\leq \|\bm{u - u_h}\|_{L^\infty([0,T];\bm{L^\infty})} + \|\bm{u}\|_{L^\infty([0,T];\bm{L^\infty})}
				\\
				&\leq C(K^*)h + \|\bm{u}\|_{L^\infty([0,T];\bm{L^\infty})}
				\\
				&\leq 2K \leq K^*.
			\end{aligned}
		\end{equation*}
		This completes the proof.
	\end{proof}

\begin{rem}
	Thanks to the estimate in Theorem \ref{fi}, Lemma \ref{tb} yields a super-convergence result:
	\begin{equation*}
		\vertiii{\left(\bm{\widetilde{u}_h - u_h}\right)(t)} \leq Ch^{k+1} \quad \forall \quad t \in (0,T].
	\end{equation*} 
	Since for $1 \leq p < \infty$, Sobolev embedding says
	\begin{equation*}
		\|\left(\bm{\widetilde{u}_h - u_h}\right)(t)\|_{\bm{L^p}} \leq C\vertiii{\left(\bm{\widetilde{u}_h - u_h}\right)(t)} \quad \forall \quad t \in (0,T],
	\end{equation*}
	we have 
	\begin{equation*}
		\|\left(\bm{\widetilde{u}_h - u_h}\right)(t)\|_{\bm{L^p}} \leq Ch^{k+1} \quad \forall \quad t \in (0,T].
	\end{equation*}
	Again as $\Omega_x \subset \R^2$, the Sobolev embedding (refer \cite[Lemma $6.4$, p. $88$]{thomee2007galerkin}) implies
	\begin{equation*}
		\begin{aligned}
			\|\left(\bm{\widetilde{u}_h - u_h}\right)(t)\|_{\bm{L^\infty}} &\leq C \left(log\left(\frac{1}{h}\right)\right)\vertiii{\left(\bm{\widetilde{u}_h - u_h}\right)(t)}
			\\
			&\leq C \left(log\left(\frac{1}{h}\right)\right)h^{k+1} \quad \forall \quad t \in (0,T].
		\end{aligned}
	\end{equation*}
\end{rem}


	\subsection{Existence and uniqueness of the discrete system \eqref{bh} and \eqref{ch}}\label{sec:exst}

	\begin{lem}\label{exst}
		There exists a unique solution $\left(\bm{u_h}, p_h, f_h\right) \in \mathcal{C}^1([0,T;,U_h\times P_h \times \mathcal{Z}_h)$ to the discrete system \eqref{bh} and \eqref{ch}. 
	\end{lem}
	
	\begin{proof}
		From \eqref{rho1}, we have that $\rho_h, \rho_h V_h \in L^2(\Omega_x)$. Since for a given $\rho_h$, \eqref{ch} leads to a system of linear algebraic equations, it is enough to prove that the corresponding homogeneous system given by 
		\begin{equation*}
			\mathcal{A}_h\left(\left(\bm{u_h}, p_h\right),\left(\bm{\phi_h}, q_h\right)\right) + \left(\rho_h\bm{u_h,\phi_h}\right)= 0
		\end{equation*} 
		has identically zero solution, i.e. $(\bm{u_h},p_h) \equiv (0,0)$.\\
		The above homogeneous system can be rewritten as 
		\begin{equation*}
			\mathcal{A}_h\left(\left(\bm{u_h}, p_h\right),\left(\bm{\phi_h}, q_h\right)\right) + \left(\rho\bm{u_h,\phi_h}\right)= \left(\left(\rho - \rho_h\right)\bm{u_h,\phi_h}\right)
		\end{equation*}
		A use of discrete inf-sup condition \eqref{dic}, the H\"older inequality, Lemma \ref{rhoh2} and Theorem \ref{fi} gives
		\begin{equation*}
			\begin{aligned}
				\alpha\vertiii{(\bm{u_h},p_h)}\, &\leq \sup_{(\bm{\phi_h},q_h) \in U_h \times P_h\setminus\{(0,0)\}}\frac{\mathcal{A}_h\left(\left(\bm{u_h},p_h\right),\left(\bm{\phi_h},q_h\right)\right) + \left(\rho\bm{u_h, \phi_h}\right)}{\vertiii{\left(\bm{\phi_h},q_h\right)}}
				\\
				&= \sup_{\bm{\phi_h} \in U_h\setminus\{0\}}\frac{\left((\rho - \rho_h)\bm{u_h, \phi_h}\right)}{\vertiii{\bm{\phi_h}}}
				\\
				&\leq \sup_{\bm{\phi_h} \in U_h\setminus\{0\}}\frac{\|\rho - \rho_h\|_{0,\mathcal{T}_h}\vertiii{\bm{u_h}}\vertiii{\bm{\phi_h}}}{\vertiii{\bm{\phi_h}}}
				\\
				&\lesssim \|f  - f_h\|_{0,\mathcal{T}_h}\vertiii{\bm{u_h}} 
				\\
				&\leq Ch^{k+1}\vertiii{\bm{u_h}} \leq Ch^{k+1}\vertiii{(\bm{u_h},p_h)}.
			\end{aligned}
		\end{equation*}
		Choosing $h$ small enough so that $\left(\alpha - Ch^{k+1}\right) > 0$ implies that $(\bm{u_h},p_h) \equiv (0,0)$. This completes the proof of existence and uniqueness of solution $(\bm{u_h},p_h)$ for a given $f_h$. 
	
	 Therefore, the solution map $f_h \mapsto \bm{u_h} = S(f_h)$ is well defined. On substituting $S(f_h)$ in place of $\bm{u_h}$ in  equation \eqref{bh} leads to a system of non-linear ODE's. Since the non-linearity is locally Lipschitz, an application of the Picard's theorem yields an existence of unique solution $f_h(t)$ for $t \in (0,t_n^*)$ for some $t_n^*$. Now, using Lemma \ref{fhbd} for boundedness of $f_h$, we can extend our solution to the complete interval $[0,T]$ using continuation arguments. This concludes the proof. 
\end{proof}

\begin{rem}
	With appropriate modifications, it is possible to extend the present analysis to $3$D Vlasov-Stokes' system with order of convergence 
	\[
	\|f(t) - f_h(t)\|_{0,\mathcal{T}_h} + \|\bm{(u - u_h)}(t)\|_{\bm{L^2}} = O(h^{k+1}) \quad \mbox{for} \quad k \geq 2.
	\]
\end{rem}
To expand it a bit in $3$D, norm comparison inequality becomes 
\begin{equation}\label{normcompeqn1}
		\|w_h\|_{L^p(T^x)} \leq Ch_x^{\frac{3}{p}-\frac{3}{q}}\|w_h\|_{L^q(T^x)}.
\end{equation}
Hence, on using \eqref{normcompeqn1}, it follows that
\begin{align*}
		\|\bm{u - u_h}\|_{\bm{L^\infty}} &\leq \|\bm{u - \mathcal{P}_xu}\|_{\bm{L^\infty}} + \|\bm{\mathcal{P}_xu - u_h}\|_{\bm{L^\infty}}
		\\
		&\lesssim h_x\|\bm{u}\|_{\bm{W^{1,\infty}}} + h^{-\frac{3}{2}}_x\|\bm{\mathcal{P}_xu - u_h}\|_{\bm{L^2}}
		\\
		&\lesssim h_x\|\bm{u}\|_{\bm{W^{1,\infty}}} + h^{-\frac{3}{2}}_x\|\bm{u - \mathcal{P}_xu}\|_{\bm{L^2}} + h^{-\frac{3}{2}}_x\|\bm{u - u_h}\|_{\bm{L^2}}
		\\
		&\lesssim h_x\|\bm{u}\|_{\bm{W^{1,\infty}}} + h^{k-\frac{1}{2}}_x\|\bm{u}\|_{\bm{H^{k+1}}} + h^{-\frac{3}{2}}_x\|\bm{u - u_h}\|_{\bm{L^2}}.
\end{align*}
Then, the following changes in \eqref{K2Estimate} as follows
	\begin{equation*}\label{K2Estimate1}
	\begin{aligned}
		|\mathcal{K}^2| &\leq Ch^k\|\bm{u - u_h}\|_{\bm{L^\infty}}\|f\|_{k+1,\Omega}\|\theta_f\|_{0,\mathcal{T}_h} 
		\\
		&+ Ch^{k+1}\left(\|\bm{u}\|_{\bm{W^{1,\infty}}}\|f\|_{k+1,\Omega} + \left(L + \|\bm{u}\|_{\bm{L^\infty}}\right)\|f\|_{k+2,\Omega}\right)\|\theta_f\|_{0,\mathcal{T}_h}
		\\
		&\leq Ch^k\left(h\|\bm{u}\|_{\bm{W^{1,\infty}}} + h^{k-\frac{1}{2}}\|\bm{u}\|_{\bm{H^{k+1}}} + h^{-\frac{3}{2}}\|\bm{u - u_h}\|_{\bm{L^2}}\right)\|f\|_{k+1,\Omega}\|\theta_f\|_{0,\mathcal{T}_h} 
		\\
		&+ Ch^{k+1}\left(\|\bm{u}\|_{\bm{W^{1,\infty}}}\|f\|_{k+1,\Omega} + \left(L + \|\bm{u}\|_{\bm{L^\infty}}\right)\|f\|_{k+2,\Omega}\right)\|\theta_f\|_{0,\mathcal{T}_h}
		\\
		&\leq Ch^k\left(h\|\bm{u}\|_{\bm{W^{1,\infty}}} + h^{k-\frac{1}{2}}\|\bm{u}\|_{\bm{H^{k+1}}} + h^{k-\frac{1}{2}}C(K^*)\right)\|f\|_{k+1,\Omega}\|\theta_f\|_{0,\mathcal{T}_h}
		\\
		&+ Ch^{k+1}\left(\|\bm{u}\|_{\bm{W^{1,\infty}}}\|f\|_{k+1,\Omega} + \left(L + \|\bm{u}\|_{\bm{L^\infty}}\right)\|f\|_{k+2,\Omega}\right)\|\theta_f\|_{0,\mathcal{T}_h}
		\\
		& + C(K^*)h^{k-\frac{3}{2}}\|\theta_f\|^2_{0,\mathcal{T}_h}
		\\
		&\leq C_1(K^*)h^{4k-1} + C_2(K^*)h^{k-\frac{3}{2}}\|\theta_f\|^2_{0,\mathcal{T}_h} + Ch^{2k+2} + C\|\theta_f\|^2_{0,\mathcal{T}_h}.
	\end{aligned}
\end{equation*}
Note the term $C_2(K^*)h^{k-\frac{3}{2}}\|\theta_f\|^2_{0,\mathcal{T}_h}$ is bounded by $C_2(K^*)\|\theta_f\|_{0,\mathcal{T}_h}^2$ for $k \geq 2$ and rest of analysis follows as in the Theorem \ref{fl2}. Hence the results of Theorem \ref{fl2} holds for $k \geq 2$ in $3$D as well.

\section{Numerical Experiments}\label{comp}
This section deals with numerical experiments using time splitting algorithm combined with proposed dG methods for the following modified Vlasov-Stokes' system: 
\begin{equation}\label{eq:continuous-models}
	\left\{
	\begin{aligned}
		\partial_t f + v\cdot \nabla_x f + \nabla_v \cdot \Big( \left( \bm{u} - v \right) f \Big) & = F(t,x,v)  \quad\mbox{ in }(0,T)\times\Omega_x\times\Omega_v,
		\\
		f(0,x,v) & = f_{0}(x,v)\qquad  \mbox{ in }\Omega_x\times\Omega_v.
	\end{aligned}
	\right.
\end{equation}
\begin{equation}\label{contstokess}
	\left\{
	\begin{aligned}
		- \Delta_x \bm{u} + \rho \bm{u} +\nabla_x p & = \rho V + G(t,x)  \quad \mbox{ in }\Omega_x,
		\\
		\nabla_x \cdot \bm{u} & = 0  \qquad  \mbox{ in }\Omega_x,
	\end{aligned}
	\right.
\end{equation}

We have used splitting algorithm for the Vlasov equation \eqref{eq:continuous-models}. To achieve this, we use the Lie-Trotter splitting. Firstly, split the equation \eqref{eq:continuous-models} as:
\begin{equation}\label{sEq:ucontstokes}
	(a)\quad	\begin{aligned}
		\partial_t f + \nabla_v\cdot\left(\left( \bm{u} - v \right)f \right) = F(t,x,v),
	\end{aligned} 
	\qquad (b) \quad \begin{aligned}
		\partial_t f + v\cdot\nabla_x f = 0.
	\end{aligned}   
\end{equation}
We solve $(a)$ of \eqref{sEq:ucontstokes} for the full time step using the given initial data $f_0$ to obtain a solution $f^*$, then $(b)$ of \eqref{sEq:ucontstokes} is solved for the full time step using $f^*$ as the initial data to obtain a solution $f$.

For a complete discretization, 
let $\{t_n\}_{n=0}^{N}$ be a uniform partition of the time interval $[0,T]$ and $t_n = n\Delta t$ with time step $\Delta t > 0$.
Let $f_h^n \in \mathcal{Z}_h, \bm{u_h}^n \in U_h, p_h^n \in P_h, \rho_h^n$ and $\rho_h^nV_h^n$ denote the approximation of $f^n = f(t_n), \bm{u}^n = \bm{u}(t_n), p^n = p(t_n), \rho^n = \rho(t_n)$ and $\rho^nV^n = \rho(t_n)V(t_n)$, respectively. 
Find $\left( \bm{u_h}^{n+1},p_h^{n+1},f_h^{n+1}\right) \in U_h \times P_h \times \mathcal{Z}_h$, for $n = 1,\cdots, N$ such that
\begin{equation}\label{fdsvusch1}
	\begin{aligned}
		a_h(\bm{u_h}^{n+1},\bm{\phi_h}) + &\,b_h(\bm{\phi_h},p_h^{n+1}) + \left(\rho_h^n\bm{u_h}^{n+1}, \bm{\phi_h}\right)
		\\
		& = \left(\rho_h^nV_h^n + G^{n+1} , \bm{\phi_h}\right) \quad \forall \quad \bm{\phi_h} \in U_h,
	\end{aligned}
\end{equation}
\begin{equation}\label{fdsvusch2}
	- b_h(\bm{u_h}^{n+1},q_h) + s_h(p_h^{n+1},q_h) = 0 \quad \forall \quad q_h \in P_h,
\end{equation}
\begin{equation}\label{fdsvusCh}
	\left(\frac{f_h^{*} - f_h^n}{\Delta t},\psi_h\right) + \mathcal{B}^x_h\left(\bm{u_h}^{n+1};f_h^n,\psi_h\right) 
	= \left(F^n,\psi_h\right) \quad \forall \quad \psi_h \in \mathcal{Z}_h,
\end{equation}
\begin{equation}\label{fdsvusBh}
	\begin{aligned}
		\left(\frac{f_h^{n+1} - f^*_h}{\Delta t}, \Psi_h\right) +\mathcal{B}^v_h(f_h^*,\Psi_h) = 0 \quad \forall \quad \Psi_h \in \mathcal{Z}_h,
	\end{aligned}
\end{equation}
 where, $ a_h(\bm{u_h}^n,\bm{\phi_h}), b_h(\bm{\phi_h},p_h^n)$ and $s_h(p^n_h,q_h)$ is defined by \eqref{aih}-\eqref{sh} at $t = t_n$. In \eqref{fdsvusCh} and \eqref{fdsvusBh} is defined as: 
 \begin{equation*}
 	\mathcal{B}^x_h(\bm{u_h}^{n+1};f_h^n,\psi_h) := \sum_{R \in \mathcal{T}_h}\mathcal{B}^x_{h,R}(\bm{u_h}^{n+1};f_h^n,\psi_h),
 \end{equation*}
 with
 \begin{equation*}\label{s-bhdef-vus}
 	\begin{aligned}
 		\mathcal{B}^x_{h,R}(\bm{u_h}^{n+1};f_h^n,\psi_{h}) & := - \int_{T^x}\int_{T^v}f_h^n\left(\bm{u_h}^{n+1} - v\right).\nabla_v\psi_h \, {\rm d}v\,{\rm d}x 
 		\\&\qquad+ \int_{T^x}\int_{\partial T^v} \reallywidehat{\left(\bm{u_h}^{n+1} - v\right)\cdot\bm{n}f_h^n}\psi_h \, {\rm d}s(v)\,{\rm d}x
 	\end{aligned}
 \end{equation*}
and 
\begin{equation*}
	\mathcal{B}^v_h(f_h^n,\Psi_h) := \sum_{R \in \mathcal{T}_h}\mathcal{B}^v_{h,R}(f_h^n,\Psi_h)
\end{equation*}
with
\begin{equation*}\label{s-chdef-vus}
	\begin{aligned}
		\mathcal{B}^v_{h,R}(f_h^n,\Psi_{h}) & :=  - \int_{T^v}\int_{T^x}f_h^n\,v.\nabla_x\Psi_{h} \, {\rm d}x\,{\rm  d}v + \int_{T^v}\int_{\partial T^x}\reallywidehat{v\cdot\bm{n}f_h^n}\Psi_{h}\, {\rm d}s(x)\,{\rm d}v
	\end{aligned}
\end{equation*}  wherein the numerical fluxes is defined by \eqref{flux} at $t = t_n$.

 For our computation, let the nodes in $\mathcal{T}_h^x$ and $\mathcal{T}_h^v$ are $\bm{x}_l$ and $\bm{v}_m$ where $l = 1, \cdots, N_x; m = 1, \cdots, N_v$, respectively. In $\mathcal{Z}_h$ space every function can be represented as $g = \sum_{l,m}g(\bm{x}_l,\bm{v}_m)L_x^l(\bm{x})L_v^m(\bm{v})$ on $R$, here $L_x^l(\bm{x})$ and $L_v^m(\bm{v})$ are the $l$-th and $m$-th Lagrangian interpolating polynomials in $T^x$ and $T^v$, respectively.  
 
 Under this setting, we can solve the equations for $f$ in the split equations $\eqref{sEq:ucontstokes} (a), (b)$ in the reduced dimensions. For example, in equation $\eqref{sEq:ucontstokes} (a)$, we fix a nodal point in $x$-direction, say $\bm{x}_l$, then solve
 \begin{equation*}
 	\begin{aligned}
 		\partial_t f(\bm{x}_l) + \nabla_v\cdot\left(\left(\bm{u}(\bm{x}_l) - v\right)f(\bm{x}_l)\right) = F(t,\bm{x}_l,v)
 	\end{aligned}
 \end{equation*}
 in $v$-direction and obtain an update of point values of $f(\bm{x}_l,\bm{v}_m)$ for all $\bm{v}_m \in \mathcal{T}^v_{h}$.
 
 Similarly, for equation $\eqref{sEq:ucontstokes} (b)$. We fix a nodal point in $v$-direction, say $\bm{v}_m$, then solve $\partial_tf(\bm{v}_m) + \bm{v}_m\cdot\nabla_x f(\bm{v}_m) = 0$ by a dG method in the $x$-direction and obtain an update of point values of $f(\bm{x}_l,\bm{v}_m)$ for all $\bm{x}_l \in \mathcal{T}^x_{h}$.

 For the plots, we use the notations $k_x$ and $k_v$ for degree of polynomials in $x$ and $v$-variables, respectively. $h_x$ and $h_v$ be the mesh sizes for $\mathcal{T}_h^x$ and $\mathcal{T}_h^v$, respectively. In the experiments, we choose $h_x = h_v = h$. The error $f - f_h, \bm{u - u_h}$ and $p - p_h$ is calculated in $L^2(\Omega), \bm{L^2}$ and $L^2(\Omega_x)$, respectively at final time $T$ and denoted by errL2f, errL2u and errL2p, respectively.

\begin{exm}\label{exm-1}
	The first example for which we test the proposed scheme. Here, the exact solution of the problem is given by
	\begin{equation*}
		\begin{aligned}
			f(t,x,y,v_1,v_2) &= \sin(\pi(x-t))\sin(\pi(y-t))e^{(-v_1^2-v_2^2)}(1-v_1^2)(1-v_2^2)(1+v_1)(1+v_2), 
			\\
			u_1(x,y) &= -\cos(2\pi x)\sin(2\pi y), \qquad u_2(x,y) = \sin(2\pi x)\cos(2\pi y),
			\\
			p(x,y) &= 2\pi\left(\cos(2\pi y) - \cos(2\pi x)\right).
		\end{aligned}
	\end{equation*} 
Note that, the initial data
\[
f(0,x,y,v_1,v_2) = \sin(\pi x)\sin(\pi y)e^{(-v_1^2-v_2^2)}(1-v_1^2)(1-v_2^2)(1+v_1)(1+v_2). 
\]
With this we calculate F and G.
\end{exm}	 

We run the simulations for the domains $\Omega_x = [0,1]^2$ and $\Omega_v = [-1,1]^2$. The penalty parameter is choose $10$ and for $k_x = 1,1$ and $k_v = 1,2$ the final time taken is $0.1$ and for $k_x = k_v = 2$ the final time is $0.01$.

\begin{figure}
	\centering
	\includegraphics[width=6.25cm]{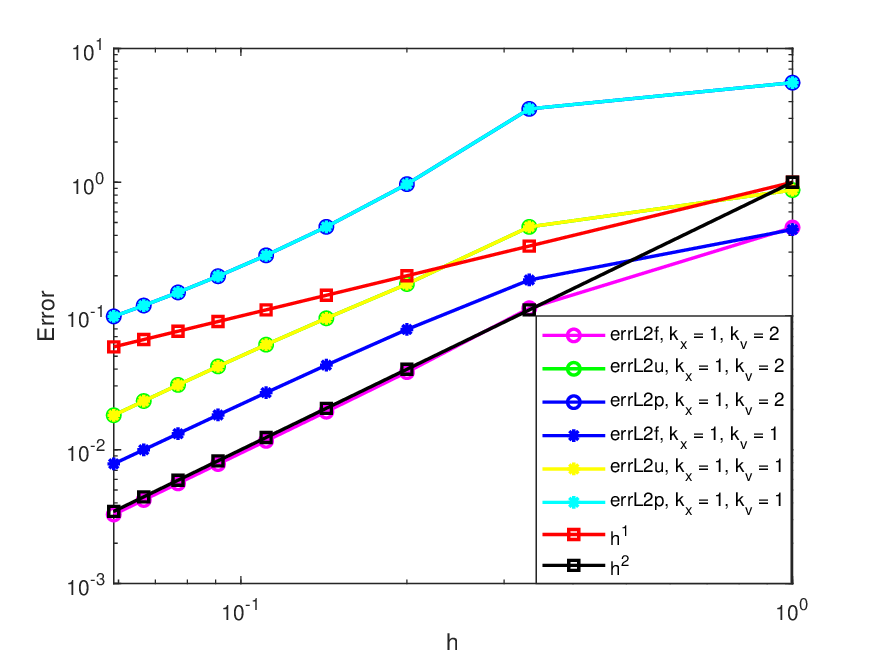}
	\includegraphics[width=6.25cm]{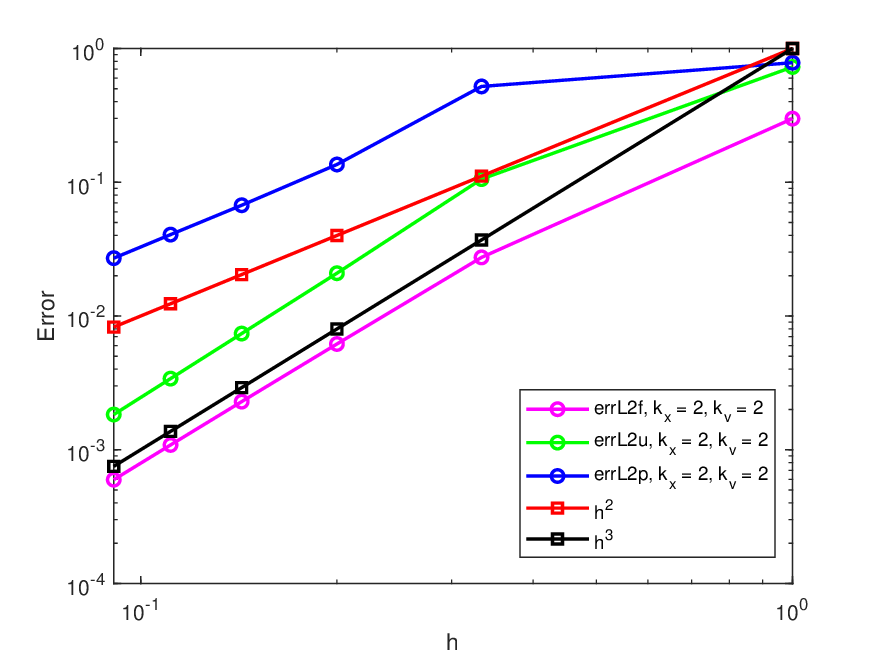}
	\caption{Convergence rate of the distribution function $f$ and the velocity $u$ for the Example \ref{exm-1};   }
	\label{fig:1}
\end{figure}

\begin{exm}\label{exm-2}
	With exact solutions of the problem given by
	\begin{equation*}
		\begin{aligned}
			f(t,x,y,v_1,v_2) &= \cos(t)\sin(2\pi x)\sin(2\pi y)e^{(-v_1^2-v_2^2)}(1-v_1^2)(1-v_2^2)(1+v_1)(1+v_2),
			\\
			u_1(x,y) &= -\cos(2\pi x)\sin(2\pi y) + \sin(2\pi y),
			\\
			u_2(x,y) &= \sin(2\pi x)\cos(2\pi y) - \sin(2\pi x),
			\\
			p(x,y) &= 2\pi\left(\cos(2\pi y) - \cos(2\pi x)\right),
		\end{aligned}
	\end{equation*} 
we calculate the initial data
\[
f(0,x,y,v_1,v_2) = \sin(2\pi x)\sin(2\pi y)e^{(-v_1^2-v_2^2)}(1-v_1^2)(1-v_2^2)(1+v_1)(1+v_2),
\] 
and the function $F$ and $G$.
\end{exm}

We run the simulations for the domains $\Omega_x = [0,1]^2$ and $\Omega_v = [-1,1]^2$. The penalty parameter is choose $10$ and for $k_x = 1,1$ and $k_v = 1,2$ the final time taken is $0.1$ and for $k_x = k_v = 2$ the final time is $0.01$.

\begin{figure}
	\centering
	\includegraphics[width=6.25cm]{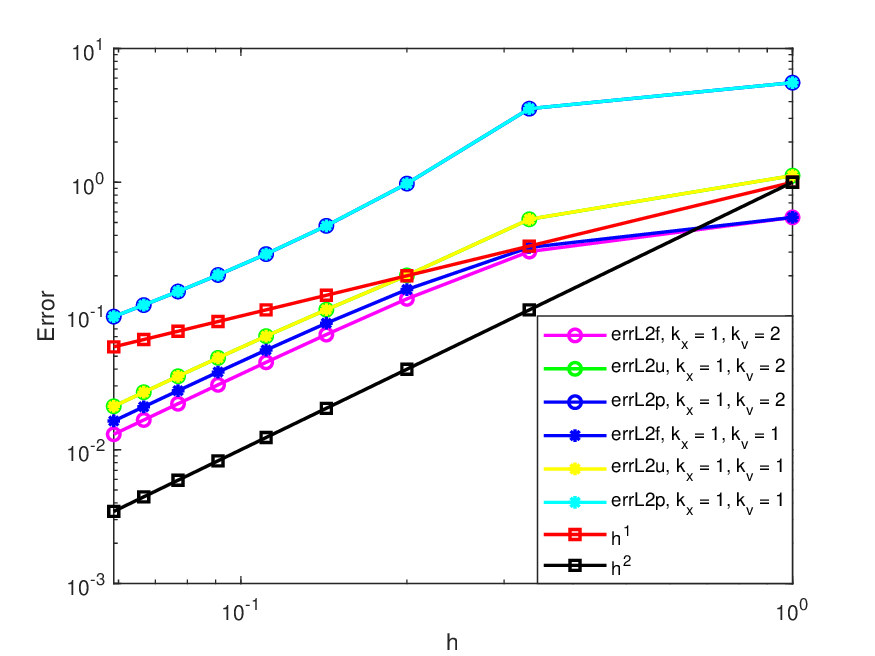}
	\includegraphics[width=6.25cm]{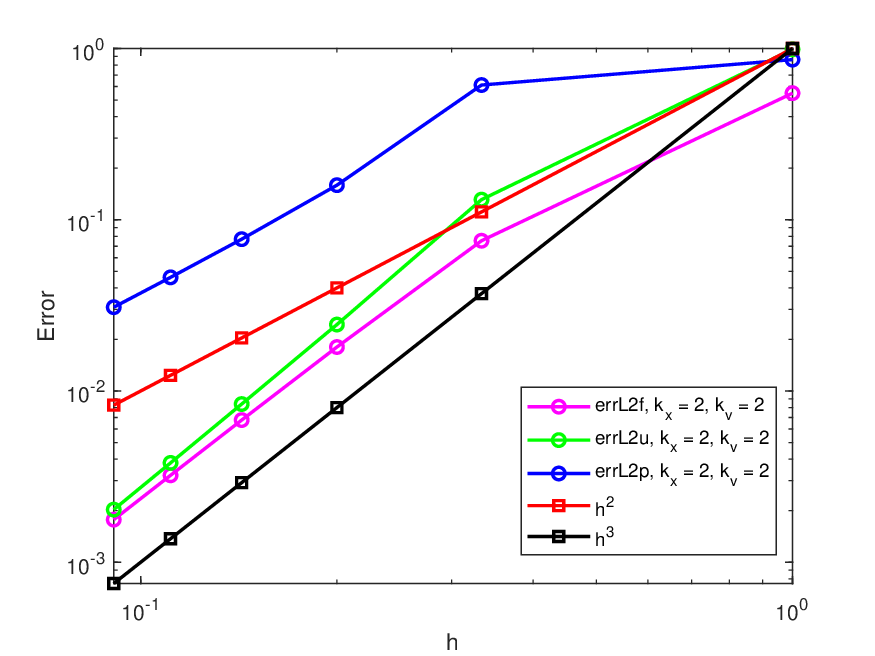}
	\caption{Convergence rate of the distribution function $f$ and the velocity $u$ for the Example \ref{exm-2};   }
	\label{fig:2}
\end{figure}

\textbf{Observations:}

\begin{itemize}
	\item From figure \ref{fig:1}(Left) and figure \ref{fig:2}(Left), we observe that by taking degree of polynomials $k_x = 1, k_v = 1$ and $k_x = 1, k_v = 2$. We obtain order of convergence for distribution function $f$ and fluid velocity $\bm{u}$ is $\min{(k_x,k_v)} + 1$ which is $2$, for pressure $p$ the order of convergence is $\min(k_x,k_v)$ which is $1$. 
	\item If we take degree of polynomials $k_x = k_v = 2$, then from figure \ref{fig:1}(Right) and figure \ref{fig:2}(Right), we can see the order of convergence for distribution function $f$ and fluid velocity $\bm{u}$ is $\min(k_x,k_v) + 1$ which is $3$, for fluid pressure $p$ the order of convergence is $\min(k_x,k_v)$ which is $2$. 
	\item  Note that it is difficult to solve the Vlasov equation in $[0,1]^2\times [-1,1]^2$ for each time $t$. But using splitting in time and with the help of a sequence of two dimensional problems, we compute the solution for each time level 
	$t = t_n$.
\end{itemize}

	\section{Conclusion}
	
This paper introduces the semi-discrete numerical method \eqref{bh}-\eqref{ch} using discontinuous Galerkin methods for the continuum problem and the conservation properties are stated in Lemma \ref{lem:distcon}. It is difficult to prove the non-negativity of the discrete density $\rho_h$. Due to this the proposed scheme is not coercive and this creates a major problem in the error estimates which is handled by a smallness condition on $h$. The optimal rate of convergence for polynomial degrees $k \geq 1$ is proved in Theorem \ref{fi}. Finally, using splitting scheme, some numerical experiments are reported.\\

	\textbf{Acknowledgements.} K.K. and H.H. thank Laurent Desvillettes for introducing them to the fluid-kinetic equations modelling the thin sprays during the Junior Trimester Program on Kinetic Theory organised at the Hausdorff Research Institute for Mathematics, Bonn. K.K. and H.H. thank the Hausdroff Institute of Mathematics, Bonn, for hosting them during the Junior Trimester program on Kinetic theory (Summer of 2019) where this work was initiated. K.K. further acknowledges the financial support of the University Grants Commission (UGC), Government of India. 
	
	\bibliography{references}
	\bibliographystyle{amsalpha}

\end{document}